\documentclass[11pt,oneside,reqno]{amsart}

\usepackage{amsmath,amssymb,amsthm,textcomp,mathtools}
\usepackage{amsfonts,graphicx}
\usepackage[mathscr]{eucal}
\usepackage{color}
\usepackage{csquotes}
\usepackage{enumitem}
\allowdisplaybreaks

\numberwithin{equation}{section}

\theoremstyle{definition}

\numberwithin{equation}{section}

\newtheorem{theorem}{\bf Theorem}[section]
\newtheorem{remark}{\bf Remark}[section]
\newtheorem{proposition}{Proposition}[section]
\newtheorem{lemma}{Lemma}[section]

\newtheorem{example}{Example}[section]
\newtheorem{definition}{Definition}[section]
\newtheoremstyle
{remarkstyle}
{}
{11pt}
{}
{}
{\bfseries}
{:}
{     }
{\thmname{#1} \thmnumber{#2} }

\theoremstyle{remarkstyle}

\begin{document}
	\title{Generalized birth-death process on Finite Lattice}
	\author[Pradeep Vishwakarma]{Pradeep Vishwakarma}
	\address{P. Vishwakarma, Department of Mathematics,
		Indian Institute of Technology Bhilai, Durg, 491002, INDIA.}
	\email{pradeepv@iitbhilai.ac.in}
	\author{Kuldeep Kumar Kataria}
	\address{K. K. Kataria, Department of Mathematics,
		 Indian Institute of Technology Bhilai, Durg, 491002, INDIA.}
	 \email{kuldeepk@iitbhilai.ac.in}

	
	\date{\today}
	
	\maketitle
	\begin{abstract}
		We consider a generalized birth-death process (GBDP) whose state space is a finite subset of a $q$-dimensional lattice. It is assumed that there can be a jump of finite step size in all possible directions such that the probability of simultaneous transition in more than one direction is zero. Such processes are of interest as their transition probability matrix is diagonalizable under suitable conditions. Thus, their $k$-step transition probabilities can be  efficiently obtained. For GBDP on two-dimensional finite grid, we obtain a sufficient and necessary condition for the vertical and horizontal transition probability matrices to commute. Later, we extend these results to the case of $q$-dimensional finite grid. We obtain the minimum number of constraints required on transition probabilities that ensure the commutation of transition probability matrices for each possible direction.
	\end{abstract}
	
\section{Introduction} A birth-death process (BDP) is a Markov process that models the number of individuals in a population over a period of time.  The BDP arises in many real life situations, for example, in ecology, epidemic, servicing systems, \textit{etc}. In BDP, at any particular instance there is a transition of unit size in either forward or backward direction with positive probability such that the probability of any other kind of transition is zero. Let $S=\{i_0,i_1,i_2,\dots\}$ be the state space of BDP, and let $p_{ij}$, $\{i,j\}\subset S$ denote its one step transition probability from state $i$ to state $j$. If we observe the process for finite time then the one step transition probability matrix $P=(p_{ij})$ is tri-diagonal. So, the matrix $P$ can be diagonalized using a family of orthogonal polynomials (see Karlin and McGregor (1955), (1959)). In such cases, the $k$-step transition probability matrix $P^{(k)}$ of BDP can be computed more efficiently.
In real life situation, it is natural to consider a discrete-time BDP whose state space is finite-dimensional integer lattice. For example, a queuing system  where there are multiple servers such that each server has different set of customers and they follow their corresponding queue. In ecology, it is of interest to study the growth of individual for each genotype in a population with different types of genotype.

Further, there exist higher-dimensional birth-death models where one of the coordinates represents the main quantity of interest and the remaining coordinates indicate the background or changing environment that affects its dynamics. This perspective is adopted in some quasi birth-death processes (see  Latouche and Ramaswami (1999), Neuts (1994)) in which the state of primary variable is referred as  level and that of subsidiary variable is referred as phase.
 These types of processes are used to model queues in random environment and the queues in which service and inter-arrival times follow phase type distributions. Although, these settings are usually in continuous time but are applicable for discrete time as well. For example, Gr\"unbaum (2008) studied these types of processes in the discrete time setting and discussed their relations to the matrix valued orthogonal polynomials. Moreover, there are extensive literature related to such processes in the fields of finance, economics, engineering, \textit{etc}. In such cases, these models are commonly referred as regime switching or stochastic volatility models (see Genon-Catalot \textit{et al.} (1999), (2000)). In practice, the discrete time and discrete state space approximations for continuous cases are frequently used in numerical computations. For more details on such processes, we refer the reader to Ephraim and Merhav (2002), Ng and Yang (2006), Yin \textit{et al.} (2004), Zhou and Yin (2003), and references therein.
	
 Like simple birth-death model, the higher-dimensional birth-death processes too have one step transitions but are not necessarily reversible. Consequently, there is no guarantee of a straightforward spectral decomposition using orthogonal polynomials. It is natural to look for cases where certain aspects of the one-dimensional theory hold. One such case occurs when it is possible to write the transition probability matrix  as the sum of pairwise commutative matrices, each representing transitions in a specific coordinate direction. 

Recently, Evans \textit{et al.} (2010) studied a birth-death process, namely, the commuting birth-death process whose state space is a finite dimensional finite grid. Using techniques from combinatorics and algebraic statistics, they derived the algebraic constraints for commuting matrices. Moreover, they have shown that the spectral decomposition methods can be utilized to obtain the $k$-step transition probabilities for a birth-death process on a two-dimensional finite grid. Also, they extended these results to the case of higher-dimensional finite grid. In this case, assuming that all  transition probabilities are positive, they obtained a unique parametrization for the commuting model. Under the same assumption, they identified the minimum number of constraints required on the transition probabilities to ensure the commutation of transition probability matrices in every possible directions. Also, they study the binomial ideal generated by the commutation condition using methods from algebraic statistics and matroid theory.

The continuous-time and discrete state space growth processes with variable step size transitions have been the topic of study for many years. For example, Di Crescenzo {\it et al.} (2016) studied a generalization of the Poisson process, namely, the generalized counting process. Unlike the Poisson process, in the generalized counting process there is a possibility of a jump of some finite step size at any point of time. Kataria and Khandakar (2021) introduced and studied a generalized pure birth process in which multiple births can happen simultaneously. The idea of multiple births  in BDP was introduced by Ramakrishnan and Srinivasan (1958) where they studied the age distribution of twins as a particular case. Doubleday (1973) introduced and studied a linear birth-death process with multiple birth in continuous-time space where he studied the distributional properties of the linear birth-death process and obtained the explicit form of the extinction probability. Recently, Vishwakarma and Kataria (2024) generalized the BDP to a case in which there is a possibility of finite step transition in an infinitesimal interval. That is, at any time instant finitely many births or finitely many deaths are possible with positive probability. However, it is assumed that the chance of simultaneous birth and death is negligible.

In this paper, we extend the results of Evans \textit{et al.} (2010) to the case of a generalized birth-death model. The model studied in their paper is a particular case of the process discussed in this work. Let us consider a birth-death process whose state space is a finite subset of finite dimensional lattice. Here, it is assumed that at any instance there can be multiple but finitely many transition in every possible direction. Also, we assume that at any point of time transition will take place in exactly one direction with some positive probability and the probability of simultaneous transition in more than one directions is zero. Moreover, the probability of self transition at any state is assumed to be zero.

  Let $\textbf{e}_i(x)\in\mathbb{N}_0^q$ be such that its $i$th coordinate is $x\in\mathbb{N}$ and zero elsewhere. Here, $\mathbb{N}_0^q$ is the set of non-negative integer valued $q$-dimensional vectors and $\mathbb{N}$ denotes the set of positive integers. Also, let  us consider a $q$-dimensional finite grid $
	G^{(q)}_{n_1,n_2,\dots,n_q}=\{0,1,\dots,n_1\}\times\{0,1,\dots,n_2\}\times\dots\times\{0,1,\dots,n_q\},
$
and let
$l_1>0$, $l_2>0$ be such that $\max\{l_1,l_2\}<\min\{n_1,n_2,\dots,n_q\}$. We consider a discrete-time birth-death process $X=\{X_k\}_{k\ge0}$ with state space $G^{(q)}_{n_1,n_2,\dots,n_q}$. We call it the generalized birth-death process (GBDP). For $\{\textbf{u}, \textbf{v}\}\subset G^{(q)}_{n_1,n_2,\dots,n_q}$, its transition probabilities are given by
\begin{equation*}
	\mathrm{Pr}\{X_{k+1}=\textbf{v}|X_k=\textbf{u}\}=\begin{cases}
		p(\textbf{u},\textbf{v}),\ \textbf{v}-\textbf{u}=\textbf{e}_i(x),\ i\in\{1,2,\dots,q\},\ x\in\{1,2,\dots,l_1\},\\
		p(\textbf{u},\textbf{v}),\ \textbf{v}-\textbf{u}=-\textbf{e}_i(y),\ i\in\{1,2,\dots,q\},\ y\in\{1,2,\dots,l_2\},\\
		0,\ \mathrm{otherwise},
	\end{cases}
\end{equation*}
where $p(\textbf{u},\textbf{v})$ denotes the one-step transition probability of $X$ from state $\textbf{u}$ to state $\textbf{v}$.

 First, we consider the case of $q=2$ and then extend the results for $q>2$. It is shown that the spectral decomposition method for a real symmetric matrix can be used to obtain the $k$-step transition probability matrix  if and only if $l_1=l_2$.

The GBDP has the following potential real life applications:\\
\noindent {(i)}
The birth-death process with multiple transitions is a potentially applicable mathematical model in biology, especially for studying cell replications of more than one type of cells. During cell division, a single cell, known as the mother cell, splits into two daughter cells. Prokaryotic cells are simple cells that divide through binary fission whereas eukaryotic cells are more complex and undergo a process of nuclear division followed by cell division.\\
\noindent {(ii)} It can be used to model a queuing
system consisting of multiple servers with different sets of customers. At each server, the customers  arrive in bulk at any point of time and
follow their corresponding queue.\\
\noindent {(iii)} It can be applied to stock market to analyze the unpredictable fluctuations in day-to-day stock prices of multiple stocks.

The paper is organized as follows: 

First, we recall some definitions and known results from the matrix theory, Markov chains and graph theory in the next section.

In Section \ref{sec2}, we consider the GBDP whose state space is a two dimensional finite grid. We assume that it can incur transitions of finite sizes in rightward, upward, downward or leftward directions with some positive probabilities. Moreover, the probability of transition in two or more directions simultaneously is zero. We obtain the system of constraints under which the horizontal and vertical transition probability matrices commute. It is shown that the spectral decomposition method can be used to obtain the $k$-step transition probability matrix if and only if the maximum size of transition is equal for each direction. We observe that if the self transition occurs with positive probability for every state then the self transition probability matrix commute with horizontal and vertical transition probability matrices if and only if the process has the same self transition probability for every state. Further, we obtain the conditions under which the transition probability matrix of GBDP is stochastic.

In Section \ref{sec3}, we consider the GBDP whose state space is a $q$-dimensional finite grid.  A system of constraints is derived under which the transition probability matrices for every direction commute with each other. And, the size of maximal set with independent constraints is obtained. 

In the concluding section, we summarize the obtained results and give some motivations for further study in this direction.

\section{Preliminaries} Here, we collect some definitions and known results that will be used in this paper. First, we fix some notations and recall some results regarding matrices (see  Horn and Johnson (2013)).

\subsection{Notations} Let $\mathbb{R}$ denote the set of real numbers and $\mathbb{N}_0=\mathbb{N}\cup\{0\}$ denote the set of non-negative integers where $\mathbb{N}$ denotes the set of natural numbers. Also, let $M_n(\mathbb{R})$, $n\ge1$ be the set of all $n\times n$ real matrices and $M_{m\times n}(\mathbb{R})$ denote the set of all $m\times n$ real matrices.

 The following result is known as the spectral theorem for real symmetric matrices:
\begin{theorem}\label{pthm1}
	A matrix $U\in M_n(\mathbb{R})$ is symmetric if and only if there exist an orthogonal matrix $O\in M_n(\mathbb{R})$ and a diagonal matrix $D\in M_n(\mathbb{R})$ such that $U=ODO^T$, where $O^T$ denotes the transpose of matrix $O$.
\end{theorem}

Also, we have the following remarks regarding the rank of a real matrix:
\begin{remark}\label{prem1}
 If $U\in M_{n\times m}(\mathbb{R})$ then $\mathrm{rank}(UU^T)=\mathrm{rank}(U)$.
\end{remark} 
\begin{remark}\label{prem2}
	Let $U\in M_n(\mathbb{R})$ be a block diagonal matrix. Then, the rank of $U$ is equal to the sum of ranks of matrices in the main diagonal of $U$.
\end{remark}
\begin{definition}(\textbf{Irreducible matrix}).
	A matrix $U\in M_n(\mathbb{R})$ is reducible if there exist a permutation matrix $S\in M_n(\mathbb{R})$ such that
	\begin{equation*}
		S^TUS=\begin{pmatrix}
			A&B\\
			0_{(n-m)\times m}&C
		\end{pmatrix},\ 1\leq m\leq n-1,
	\end{equation*}
	where $A$, $B$ and $C$ are real matrices of appropriate orders and $0_{(n-m)\times m}$ is a zero matrix of order $(n-m)\times m$.
	
	A matrix $U\in M_n(\mathbb{R})$ is called irreducible if it is not reducible.
\end{definition}

The following result for a real non-negative and irreducible matrix is known as the Perron-Frobenius theorem:
\begin{theorem}\label{pthm2}
	Let $U\in M_n(\mathbb{R})$ be an irreducible and non-negative matrix, and let $\rho(U)$ be the eigenvalue of $U$ with largest absolute value. Then, $\rho(U)>0$ and there exist a unique positive vector $\textbf{v}=(v_1,v_2,\dots,v_n)\in\mathbb{R}^n$ such that $U\textbf{v}=\rho(U)\textbf{v}$ and $v_1+v_2+\dots+v_n=1$.
\end{theorem}

 The following result regarding the reversible Markov chain will be used (see Norris (1998)):
 \begin{theorem}\label{pthm3}
 	Let $I$ be a countable set, and let $\{X_n\}_{n\ge0}$ be a Markov chain with state space $I$, transition probability matrix $P=(p_{ij})$, $i,\,j\in I$ and initial distribution $\Pi=(\pi_i)_{i\in I}$. Then, the following are equivalent:\\
 	\noindent $(i)$ $\{X_n\}_{n\ge0}$ is reversible;\\
 	\noindent $(ii)$ $P$ and $\Pi$ are in detailed balance, that is, $\pi_ip_{ij}=\pi_jp_{ji}$ for all $i,j$.
 \end{theorem}
 
 Next, we collect some definitions from graph theory (Godsil and Royle (2001)). 
 \begin{definition}(\textbf{Adjacency matrix})
 	Let $\mathcal{G}$ be a graph with vertex set $V=\{v_1,v_2,\dots,v_m\}$. Then, the adjacency matrix of $\mathcal{G}$ is a $m\times m$ matrix whose $(i,j)$ entry is $1$ if there is an edge between $v_i$ and $v_j$ otherwise it is $0$. 
 \end{definition} 
 \begin{definition}{(\textbf{Incidence matrix})}
 	The incidence matrix of a graph $\mathcal{G}$ is a binary matrix with rows and columns index by its vertices and edges, respectively, and its $(i,e)$ entry is $1$ if and only if the vertex $i$ is in the edge $e$.
 \end{definition}
  The Laplacian of a graph $\mathcal{G}$ is defined as $L(\mathcal{G})=KK^T$ where $K$ is the incidence matrix of $\mathcal{G}$.
  
  The following result regarding the rank of Laplacian of a graph will be used (Godsil and Royle (2001)):
 \begin{lemma}\label{plemm1}
 	Let $\mathcal{G}$ be a graph with $n$ vertices and $a$ many connected components. Then, the rank of $L(\mathcal{G})$ is $n-a$.
 \end{lemma}
\section{GBDP on 2-dimensional grid}\label{sec2}
 Let us consider a generalized discrete-time Markov chain $X=\{X_k\}_{k\ge0}$ whose state space is a two dimensional grid $G^{(2)}_{m-1,n-1}=\{0,1,\dots,m-1\}\times\{0,1,\dots,n-1\}$, $m\ge2$, $n\ge2$. Let $l_1\ge1$ and $l_2\ge1$ be such that $\mathrm{max}\{l_1,l_2\}\leq\mathrm{min}\{m-1,n-1\}$. It is assumed that at any instance the chain can make jumps of sizes $1,2,\dots,l_1$ either rightward or upward. Further, it can make jumps of sizes $1,2,\dots,l_2$ either leftward or downward. Also, we assume that the probability of simultaneous transitions in vertical and horizontal directions is zero. Moreover, the probability of any other kind of jumps is zero. That is, if we create a graph structure of the state space $G^{(2)}_{m-1,n-1}$ then two vertices $(i,j)$ and $(i',j')$ of $G^{(2)}_{m-1,n-1}$ are connected via an edge if and only if either $|i-i'|\leq \max\{l_1,l_2\}$ with $j-j'=0$ or $|j-j'|\leq \max\{l_1,l_2\}$ with $i-i'=0$.
For $x\in\{1,2,\dots,l_1\}$, $y\in\{1,2,\dots,l_2\}$ and $(i,j)\in G^{(2)}_{m-1,n-1}$, we denote the transition probabilities of $X$ as follows: 
\begin{align*}
	r_{i,j}(x)&=\mathbb{P}\{X_{k+1}=(i+x,j)|X_k=(i,j)\},\\
		l_{i,j}(y)&=\mathbb{P}\{X_{k+1}=(i-y,j)|X_k=(i,j)\},\\
			u_{i,j}(x)&=\mathbb{P}\{X_{k+1}=(i,j+x)|X_k=(i,j)\}
\end{align*}
and
\begin{equation*}
	d_{i,j}(y)=\mathbb{P}\{X_{k+1}=(i,j-y)|X_k=(i,j)\},
\end{equation*}
	where $r_{i,j}(x)=0$, $l_{i,j}(y)=0$, $u_{i,j}(x)=0$ and $d_{i,j}(y)=0$ whenever $i+x>m-1$, $i-y<0$, $j+x>n-1$ and $j-y<0$,  respectively. These transition probabilities are such that $\sum_{x=1}^{l_1}(r_{i,j}(x)+u_{i,j}(x))+\sum_{y=1}^{l_2}(l_{i,j}(y)+d_{i,j}(y))\leq1$ for all $(i,j)\in G^{(2)}_{m-1,n-1}$. It's a possibility that this inequality is strict. Usually the strict inequality is justified by an absorbing state $(i_s,j_s)$ such that $$	\mathbb{P}\{X_{k+1}=(i_s,j_s)|X_k=(i,j)\}=1-\sum_{x=1}^{l_1}(r_{i,j}(x)+u_{i,j}(x))-\sum_{y=1}^{l_2}(l_{i,j}(y)+d_{i,j}(y)).$$
	
	Let $P^h$ and $P^v$  be the horizontal and vertical transition probability matrix of $X$, respectively. Then, its transition probability matrix $P$ can be obtained as $P=P^h+P^v$ whose dimension is $mn\times mn$. Further, if $P^h$ and $P^v$ commute then for $n\ge1$, the $n$-step transition probability matrix of $X$ is $
		P^{(n)}=(P^h+P^v)^n$ which can be procured using the binomial theorem.
	Suppose that the self transitions are allowed for every states. Then, the self transition probability matrix $D$ is a diagonal matrix. In this case, the transition probability matrix of $X$ is given by $P=P^h+P^v+D$. Moreover, if $D$ is a scalar matrix, that is, the self transition probabilities are same for all $(i,j)\in G^{(2)}_{m-1,n-1}$ then the $k$-step transition probability matrix of $X$ is given by
	\begin{equation*}
		P^{(k)}=({P^h}+{P^v}+D)^k=\sum_{a+b+c=k}\frac{k!}{a!b!c!}(P^h)^a(P^v)^bD^c.
	\end{equation*}
	
	The following result provides a necessary and sufficient condition for which the matrices $P^h$ and $P^v$ commute.
	\begin{theorem}\label{thm1}
		 Let $s=(s_1,s_2)\in G^{(2)}_{m-1,n-1}$, $t=(t_1,t_2)\in G^{(2)}_{m-1,n-1}$ and $(i,j)\in \{0,1,\dots,m-1\}\times\{0,1,\dots,n-1\}$. Then, the transition probability matrices $P^h=(p^h_{s,t})$ and $P^v=(p^v_{s,t})$ commute with each other if and only if the following conditions hold:
		\begin{equation*}
			\begin{aligned}
				u_{i,j}(x)r_{i,j+x}(x')&=r_{i,j}(x')u_{i+x',j}(x),\\
				d_{i,j+y}(y)r_{i,j}(x)&=r_{i,j+y}(x)d_{i+x,j+y}(y),\\
				d_{i+y,j+y'}(y')l_{i+y,j}(y)&=l_{i+y,j+y'}(y)d_{i,j+y'}(y')
			\end{aligned}
		\end{equation*}
		and 
		\begin{equation*}
			u_{i+y,j}(x)l_{i+y,j+x}(y)=l_{i+y,j}(y)u_{i,j}(x),
		\end{equation*}
		where $x\in\{1,2,\dots,l_1\}$, $x'\in\{1,2,\dots,l_1\}$, $y\in\{1,2,\dots,l_2\}$ and $y'\in\{1,2,\dots,l_2\}$.
	\end{theorem}
	\begin{proof}
		Let us consider the following order on the state space $G^{(2)}_{m-1,n-1}$: $I=\{(0,0),(0,1),\dots,(0,n-1),(1,0), (1,1),\dots,(1,n-1),\dots,(m-1,0),(m-1,1),\dots,(m-1,n-1)\}$. Also, let the $s$th row and $t$th column of matrices $P^h$ and $P^v$ are index by the $s$th and $t$th element of $I$, respectively.
		Then, for $x\in\{1,2,\dots.l_1\}$ and $y\in\{1,2,\dots,l_2\}$, the horizontal transition probabilities of $X$ are given by
		\begin{equation}\label{tph}
			p^h_{s,t}=\begin{cases}
				r_{s_1,s_2}(x),\ t-s=(x,0),\vspace{0.1cm}\\
				l_{s_1,s_2}(y),\ t-s=(-y,0),\vspace{0.1cm}\\
				0,\ \mathrm{otherwise},
			\end{cases}
		\end{equation}
		and its vertical transition probabilities are given by
		\begin{equation}\label{tpv}
			p^v_{s,t}=\begin{cases}
				u_{s_1,s_2}(x),\ t-s=(0,x),\vspace{0.1cm}\\
				d_{s_1,s_2}(y),\ t-s=(0,-y),\vspace{0.1cm}\\
				0,\ \mathrm{otherwise},
			\end{cases}
		\end{equation}
		where $t-s=(t_1-s_1,t_2-s_2)$. So, the commuting condition $P^hP^v=P^vP^h$ implies that 
		\begin{equation}\label{Commute}
			\sum_{t'\in I}p^h_{s,t'}p^v_{t',t}=\sum_{t'\in I}p^v_{s,t'}p^h_{t',t}
		\end{equation}
		for all pairs $(s,t)\in I\times I$. Thus, by using (\ref{tph}) and (\ref{tpv}) in (\ref{Commute}), we get the required result.
	\end{proof}
	
	\begin{remark}
		For $l_1=l_2=1$, the graph of $G^{(2)}_{m-1,n-1}$ is a $m\times n$ grid. In this case, the process $X$ reduces to the commuting birth-death process introduced and studied by Evans \textit{et al.} (2010). So, $P^h$ and $P^v$ commute if and only if the following conditions hold:
		\begin{equation*}
			\begin{aligned}
				u_{i,j}(1)r_{i,j+1}(1)&=r_{i,j}(1)u_{i+1,j}(1),\\  d_{i,j+1}(1)r_{i,j}(1)&=r_{i,j+1}(1)d_{i+1,j+1}(1),\\
				d_{i+1,j+1}(1)l_{i+1,j}(1)&=l_{i+1,j+1}(1)d_{i,j+1}(1)
			\end{aligned}
		\end{equation*}
		and
		\begin{equation*}
			u_{i+1,j}(1)l_{i+1,j+1}(1)=l_{i+1,j}(1)u_{i,j}(1).
		\end{equation*}
		Thus, the probability of moving from one corner of $G^{(2)}_{m-1,n-1}$ to its diagonally opposite corner in two steps is same for every path.
	\end{remark}

	\begin{example}\label{exp1}		
		For $m=n=3$, the graph of $G^{(2)}_{2,2}$ is given in Figure \ref{fig1}. 		
		\begin{figure}[ht!]
			\includegraphics[width=6cm]{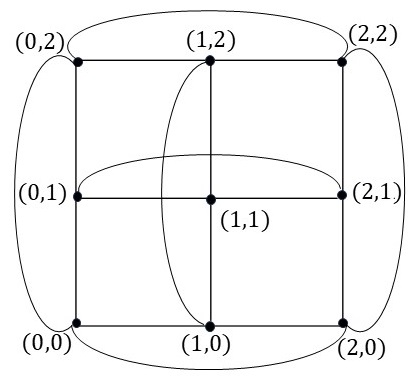}
			\caption{ $2\times2$ grid with nine vertices}\label{fig1}
		\end{figure}
		
		For $l_1=l_2=2$, the $9\times 9$ horizontal and vertical transition probability  matrices are given as follows:
		\begin{equation*}\label{ph}
			P^h=\ \begin{array}{c|ccccccccc}
				& (0,0) & (0,1) & (0,2)&(1,0)&(1,1)&(1,2)&(2,0)&(2,1)&(2,2) \\ \hline
				(0,0) & 0 & 0 & 0&r_{0,0}(1)&0&0&r_{0,0}(2)&0&0 \\
				(0,1) & 0 & 0 & 0&0&r_{0,1}(1)&0 &0&r_{0,1}(2)&0 \\
				(0,2) & 0 & 0 & 0&0&0&r_{0,2}(1)&0&0&r_{0,2}(2)\\
				(1,0)&l_{1,0}(1)&0&0&0&0&0&r_{1,0}(1)&0&0\\
				(1,1)&0&l_{1,1}(1)&0&0&0&0&0&r_{1,1}(1)&0\\
				(1,2)&0&0&l_{1,2}(1)&0&0&0&0&0&r_{1,2}(1)\\
				(2,0)&l_{2,0}(2)&0&0&l_{2,0}(1)&0&0&0&0&0\\
				(2,1)&0&l_{2,1}(2)&0&0&l_{2,1}(1)&0&0&0&0\\
				(2,2)&0&0&l_{2,2}(2)&0&0&l_{2,2}(1)&0&0&0
			\end{array}
		\end{equation*}
		and
	\begin{equation*}\label{pv}
		P^v=\ \begin{array}{c|ccccccccc}
			& (0,0) & (0,1) & (0,2)&(1,0)&(1,1)&(1,2)&(2,0)&(2,1)&(2,2) \\ \hline
			(0,0) & 0 & u_{0,0}(1) & u_{0,0}(2)&0&0&0&0&0&0 \\
			(0,1) & d_{0,1}(1) & 0 & u_{0,1}(1)&0&0&0 &0&0&0 \\
			(0,2) & d_{0,2}(2) & d_{0,2}(1) & 0&0&0&0&0&0&0\\
			(1,0)&0&0&0&0&u_{1,0}(1)&u_{1,0}(2)&0&0&0\\
			(1,1)&0&0&0&d_{1,1}(1)&0&u_{1,1}(1)&0&0&0\\
			(1,2)&0&0&0&d_{1,2}(2)&d_{1,2}(1)&0&0&0&0\\
			(2,0)&0&0&0&0&0&0&0&u_{2,0}(1)&u_{2,0}(2)\\
			(2,1)&0&0&0&0&0&0&d_{2,1}(1)&0&u_{2,1}(1)\\
			(2,2)&0&0&0&0&0&0&d_{2,2}(2)&d_{2,2}(1)&0
		\end{array}
	\end{equation*}		
		  If  all the transition probabilities $u_{i,j}(x)$, $r_{i,j}(x)$, $d_{i,j}(y)$ and $l_{i,j}(y)$ in $P^h$ and $P^v$ are positive then the commuting conditions given in Theorem \ref{thm1} for $m=n=3$ and $l_1=l_2=2$ is equivalent to the following constraints:		 
			\begin{align}
				\mathrm{rank}&\begin{pmatrix}
					r_{0,0}(1)& d_{1,1}(1)&l_{1,1}(1)&u_{0,0}(1)\vspace{0.1cm}\\
					r_{0,1}(1)&d_{0,1}(1)&l_{1,0}(1)&u_{1,0}(1)
				\end{pmatrix}=1,\nonumber\\
				\mathrm{rank}&\begin{pmatrix}
					r_{1,0}(1)&d_{2,1}(1)&l_{2,1}(1)&u_{1,0}(1)\vspace{0.1cm}\\
					r_{1,1}(1)&d_{1,1}(1)&l_{2,0}(1)&u_{2,0}(1)
				\end{pmatrix}=1,\nonumber\\
				\mathrm{rank}&\begin{pmatrix}
					r_{0,1}(1)&d_{1,2}(1)&l_{1,2}(1)&u_{0,1}(1)\vspace{0.1cm}\\
					r_{0,2}(1)&d_{0,2}(1)&l_{1,1}(1)&u_{1,1}(1)			\end{pmatrix}=1,\nonumber\\
				\mathrm{rank}&\begin{pmatrix}
					r_{1,1}(1)&d_{2,2}(1)&l_{2,2}(1)&u_{1,1}(1)\vspace{0.1cm}\\
					r_{1,2}(1)&d_{1,2}(1)&l_{2,1}(1)&u_{2,1}(1)
				\end{pmatrix}=1,\nonumber\\
				\mathrm{rank}&\begin{pmatrix}
					r_{0,0}(1)&d_{1,2}(2)&l_{1,2}(1)&u_{0,0}(2)\vspace{0.1cm}\\
					r_{0,2}(1)&d_{0,2}(2)&l_{1,0}(1)&u_{1,0}(2)
				\end{pmatrix}=1,\label{matrices}\\
				\mathrm{rank}&\begin{pmatrix}
					r_{1,0}(1)&d_{2,2}(2)&l_{2,2}(1)&u_{1,0}(2)\vspace{0.1cm}\\
					r_{1,2}(1)&d_{1,2}(2)&l_{2,0}(1)&u_{2,0}(2)
				\end{pmatrix}=1,\nonumber\\
				\mathrm{rank}&\begin{pmatrix}
					r_{0,1}(2)&d_{0,1}(1)&l_{2,0}(2)&u_{2,0}(1)\vspace{0.1cm}\\
					r_{0,0}(2)&d_{2,1}(1)&l_{2,1}(2)&u_{0,0}(1)
				\end{pmatrix}=1,\nonumber\\
				\mathrm{rank}&\begin{pmatrix}
					r_{0,2}(2)&d_{0,2}(1)&l_{2,1}(2)&u_{2,1}(1)\vspace{0.1cm}\\
					r_{0,1}(2)&d_{2,2}(1)&l_{2,2}(2)&u_{0,1}(1)
				\end{pmatrix}=1\nonumber
			\end{align}
			and
			\begin{equation*}
				\mathrm{rank}\begin{pmatrix}
					r_{0,0}(2)&d_{2,2}(2)&l_{2,2}(2)&u_{0,0}(2)\vspace{0.1cm}\\
					r_{0,2}(2)&d_{0,2}(2)&l_{2,0}(2)&u_{2,0}(2)
				\end{pmatrix}=1.
			\end{equation*}
		That is, for $m=n=3$ and $l_1=l_2=2$, the thirty six equations corresponding to the conditions in Theorem \ref{thm1} are among minors of order $2$ of the nine matrices given in (\ref{matrices}).
		
		 Let $\{h_1, h_2, h_3, v_1, v_2, v_3\}$ and $\{b_{i,j},\ (i,j)\in G^{(2)}_{2,2}\}$ be two collections of positive constants. Then, the commuting variety of transition probability matrices $P^h$ and $P^v$, that is, $P^hP^v=P^vP^h$  has the following parametrization:
		\begin{equation}\label{matricesrank}
			\begin{aligned}
				&r_{0,0}(1)=h_1\frac{b_{0,0}}{b_{1,0}},\  r_{0,0}(2)=h_3\frac{b_{0,0}}{b_{2,0}},\  r_{0,1}(1)=h_1\frac{b_{0,1}}{b_{1,1}},\  r_{0,1}(2)=h_3\frac{b_{0,1}}{b_{2,1}},\\ &r_{0,2}(1)=h_1\frac{b_{0,2}}{b_{1,2}},\ r_{0,2}(2)=h_3\frac{b_{0,2}}{b_{2,2}},\ r_{1,0}(1)=h_2\frac{b_{1,0}}{b_{2,0}},\  r_{1,1}(1)=h_2\frac{b_{1,1}}{b_{2,1}},\\
				&r_{1,2}(1)=h_2\frac{b_{1,2}}{b_{2,2}},\ l_{1,0}(1)=h_1\frac{b_{1,0}}{b_{0,0}},\
				l_{1,1}(1)=h_1\frac{b_{1,1}}{b_{0,1}},\  l_{1,2}(1)=h_1\frac{b_{1,2}}{b_{0,2}},\\
				&l_{2,0}(1)=h_2\frac{b_{2,0}}{b_{1,0}},\  l_{2,1}(1)=h_2\frac{b_{2,1}}{b_{1,1}},\  l_{2,2}(1)=h_2\frac{b_{2,2}}{b_{1,2}},\  l_{2,0}(2)=h_3\frac{b_{2,0}}{b_{0,0}},\\
				&l_{2,1}(2)=h_3\frac{b_{2,1}}{b_{0,1}},\ l_{2,2}(2)=h_3\frac{b_{2,2}}{b_{0,2}},\ u_{0,0}(1)=v_1\frac{b_{0,0}}{b_{0,1}},\ u_{0,0}(2)=v_3\frac{b_{0,0}}{b_{0,2}},\\
				& u_{0,1}(1)=v_2\frac{b_{0,1}}{b_{0,2}},\  u_{1,0}(1)=v_1\frac{b_{1,0}}{b_{1,1}},\ u_{1,0}(2)=v_3\frac{b_{1,0}}{b_{1,2}},\ u_{1,1}(1)=v_2\frac{b_{1,1}}{b_{1,2}},\\
				&u_{2,0}(1)=v_1\frac{b_{2,0}}{b_{2,1}},\ u_{2,0}(2)=v_3\frac{b_{2,0}}{b_{2,2}},\  u_{2,1}(1)=v_2\frac{b_{2,1}}{b_{2,2}},\  d_{0,1}(1)=v_1\frac{b_{0,1}}{b_{0,0}},\\ 
				&d_{0,2}(1)=v_2\frac{b_{0,2}}{b_{0,1}},\ d_{1,1}(1)=v_1\frac{b_{1,1}}{b_{1,0}},\ d_{1,2}(1)=v_2\frac{b_{1,2}}{b_{1,1}},\  d_{2,1}(1)=v_1\frac{b_{2,1}}{b_{2,0}},\\
				&d_{2,2}(1)=v_2\frac{b_{2,2}}{b_{2,1}},\ d_{0,2}(2)=v_3\frac{b_{0,2}}{b_{0,0}},\ d_{1,2}(2)=v_3\frac{b_{1,2}}{b_{1,0}},\ d_{2,2}(2)=v_3\frac{b_{2,2}}{b_{2,0}}.
			\end{aligned}
		\end{equation}
		It is important to note that if some transition probabilities are zero then the corresponding conditions given in Theorem \ref{thm1} will satisfy even if the condition (\ref{matricesrank}) does not hold. The parametrization given in (\ref{matricesrank}) can be rewritten in the matrix form as follows:
		\begin{equation}\label{matrixrep}
			P^h=BA^hB^{-1}\ \ \mathrm{and}\ \ P^v=BA^vB^{-1},
		\end{equation}
		where $A^h$, $A^v$ and $B$ are $9\times 9$ matrices given by	
			\begin{equation*}
				A^h=\ \begin{array}{c|ccccccccc}
					& (0,0) & (0,1) & (0,2)&(1,0)&(1,1)&(1,2)&(2,0)&(2,1)&(2,2) \\ \hline
					(0,0) & 0 & 0 & 0&h_1&0&0&h_3&0&0 \\
					(0,1) & 0 & 0 & 0&0&h_1&0 &0&h_3&0 \\
					(0,2) & 0 & 0 & 0&0&0&h_1&0&0&h_3\\
					(1,0)&h_1&0&0&0&0&0&h_2&0&0\\
					(1,1)&0&h_1&0&0&0&0&0&h_2&0\\
					(1,2)&0&0&h_1&0&0&0&0&0&h_2\\
					(2,0)&h_3&0&0&h_2&0&0&0&0&0\\
					(2,1)&0&h_3&0&0&h_2&0&0&0&0\\
					(2,2)&0&0&h_3&0&0&h_2&0&0&0
				\end{array},
		\end{equation*}
			\begin{equation*}
				A^v=\ \begin{array}{c|ccccccccc}
					& (0,0) & (0,1) & (0,2)&(1,0)&(1,1)&(1,2)&(2,0)&(2,1)&(2,2) \\ \hline
					(0,0) & 0 & v_1 & v_3&0&0&0&0&0&0 \\
					(0,1) & v_1 & 0 & v_2&0&0&0 &0&0&0 \\
					(0,2) & v_3 & v_2 & 0&0&0&0&0&0&0\\
					(1,0)&0&0&0&0&v_1&v_3&0&0&0\\
					(1,1)&0&0&0&v_1&0&v_2&0&0&0\\
					(1,2)&0&0&0&v_3&v_2&0&0&0&0\\
					(2,0)&0&0&0&0&0&0&0&v_1&v_3\\
					(2,1)&0&0&0&0&0&0&v_1&0&v_2\\
					(2,2)&0&0&0&0&0&0&v_3&v_2&0
				\end{array},
		\end{equation*}
		and
		\begin{equation*}
			B=\ \begin{array}{c|ccccccccc}
				& (0,0) & (0,1) & (0,2)&(1,0)&(1,1)&(1,2)&(2,0)&(2,1)&(2,2) \\ \hline
				(0,0) & b_{0,0} & 0 & 0&0&0&0&0&0&0 \\
				(0,1) & 0 & b_{0,1} & 0&0&0&0 &0&0&0 \\
				(0,2) & 0 & 0 &b_{0,2}&0&0&0&0&0&0\\
				(1,0)&0&0&0&b_{1,0}&0&0&0&0&0\\
				(1,1)&0&0&0&0&b_{1,1}&0&0&0&0\\
				(1,2)&0&0&0&0&0&b_{1,2}&0&0&0\\
				(2,0)&0&0&0&0&0&0&b_{2,0}&0&0\\
				(2,1)&0&0&0&0&0&0&0&b_{2,1}&0\\
				(2,2)&0&0&0&0&0&0&0&0&b_{2,2}
			\end{array},
		\end{equation*}
		respectively.
		
		Note that $B$ is a $9\times9$ diagonal matrix, $A^v$ is a block diagonal matrix with $3$ identical $3\times3$ symmetric and tri-diagonal block matrices $U^v$ and $A^h$ is a block diagonal matrix with $3$ identical $3\times3$ symmetric and tri-diagonal block matrices $U^h$. For matrix $A^h$, this point is evident by reordering the rows and columns of $A^h$ as $(0,0),\,(1,0),\,(2,0),\,(0,1),\,(1,1),\,(2,1),\,(0,2),\,(1,2),\,(2,2)$. Then, the matrices $U^h$ and $U^v$ satisfy the following conditions:
		\begin{equation*}\label{reo1}
			{A^h}_{(i,j),(i',j')}=\begin{cases}
				{U^h}_{(i,i')},\ j=j',\\
				0,\ \mathrm{otherwise}
			\end{cases}
		\end{equation*} 
		and
		\begin{equation*}\label{reo2}
			{A^v}_{(i,j),(i',j')}=\begin{cases}
				{U^v}_{(j,j')},\ i=i',\\
				0,\ \mathrm{otherwise},
			\end{cases}
		\end{equation*}
		where ${A^h}_{(i,j),(i',j')}$ denote the ${((i,j),(i',j'))}$ entry of the matrix $A^h$. Here, $U^h_{(i,i')}$  is an entry of block matrices $U^h$ corresponding to the row indexed by $(i,j)$ and the column indexed by $(i',j)$ of $A^h$, and  $U^v_{(j,j')}$ is an entry of $U^v$ corresponding to the row indexed by $(i,j)$ and  the column indexed by $(i,j')$ of $A^v$.
	\end{example}

A similar parametrization can be established for the transition probability matrices of GBDP with state space $G^{(2)}_{m-1,n-1}$. In this case, we will get a $m\times n$ diagonal matrix $B$, a block diagonal matrix $A^h$ with $n$ identical $m\times m$ symmetric and tri-diagonal block matrices $U^h$, and a second block diagonal matrix $A^v$ with $m$ identical $n\times n$ symmetric and tri-diagonal block matrices $U^v$.

For $i =0, 1, \dots, m-1$ and $j=0, 1, \dots, n-1$, let $\lambda^h_i$  and $\lambda^v_j$ be the eigenvalues of $U^h$ and $U^v$ whose associated orthonormalized eigenvectors are $w^h_i$ and $w^v_j$, respectively.  Then, from Theorem \ref{pthm1},  the entries of $k$th power of block matrices $U^h$ and $U^v$ corresponding to the entries $U^h_{(i,i')}$ and $U^v_{(j,j')}$ are given by
\begin{equation}\label{spectdec1}
	({U^h})^{k}_{(i,i')}=\sum_{r=0}^{m-1}(\lambda^h_r)^kw^h_r(i)w^h_r(i')
\end{equation}
and
\begin{equation}\label{spectdec2}
	({U^v})^{k}_{(j,j')}=\sum_{r=0}^{n-1}(\lambda_r^v)^kw_r^v(j)w_r^v(j'),
\end{equation}
respectively. Here, $w_r^h(i)$, $i=1,2,\dots,m$ and $w^v_s(j)$, $j=1,2,\dots, n$ are the $i$th and $j$th components of eigenvectors $w^h_r$, $r=0,1,\dots,m-1$ and $w^v_s$, $s=0,1,\dots,n-1$, respectively.

	Recall that the transition probability matrices $P^h$ and $P^v$ commute. So, on using (\ref{matrixrep}), we can write the $k$-step transition probability matrix of the process $X$ in the following form:
	\begin{align*}
		P^{(k)}&=\sum_{r=0}^{k}\binom{k}{r}(P^h)^r(P^v)^{k-r}\\
		&=\sum_{r=0}^{k}\binom{k}{r}\left(B{A^h}B^{-1}\right)^r\left(B{A^v}B^{-1}\right)^{k-r}\\
		&=B\sum_{r=0}^{k}\binom{k}{r}{(A^h)}^r{(A^v)}^{k-r}B^{-1}.
	\end{align*}
On using (\ref{spectdec1}) and (\ref{spectdec2}), we get the $((i,j),(i',j'))$ entry of the $k$-step transition probability matrix $P$ as follows:
\begin{equation*}
	{p^{(k)}}_{(i,j),(i',j')}=b_{i,j}\left(\sum_{r=0}^{m-1}\sum_{s=0}^{n-1}(\lambda_r^h+\lambda^v_s)^kw^h_r(i)w^h_r(i')w^v_s(j)w^v_s(j')\right)b_{i',j'}^{-1}.
\end{equation*}
Therefore, for $l_1=l_2$, the problem of finding the $k$-step transition probabilities reduces to finding the eigenvalues-eigenvectors of some lower order symmetric and tri-diagonal matrices.

\begin{remark}
	Let $D$ denote the self transition probability matrix of $X$ and $I$ be the identity matrix of order equal to that of $D$. If the self transition is allowed for each state then $D$ commute with $P^h$ and $P^v$ if and only if it is a scalar matrix, that is, there exist a constant $\alpha\in(0,1)$ such that $D=\alpha I$. Thus, the $k$-step transition probability matrix of $X$ can be obtained as follows: 
	\begin{equation*}
		P^{(k)}=\sum_{a+b+c=k}\frac{k!}{a!b!c!}\alpha^a(P^h)^b(P^v)^c=B\sum_{a+b+c=k}\frac{k!}{a!b!c!}\alpha^a{(A^h)}^b{(A^v)}^cB^{-1}.
	\end{equation*}
	Therefore, the $((i,j),(i',j'))$ entry of matrix $P^{(k)}$ is given by
	\begin{equation*}
		{p^{(k)}}_{(i,j),(i',j')}=b_{i,j}\left(\sum_{r=0}^{m-1}\sum_{s=0}^{n-1}(\alpha+\lambda_r^h+\lambda^v_s)^kw^h_r(i)w^h_r(i')w^v_s(j)w^v_s(j')\right)b_{i',j'}^{-1}.
	\end{equation*}
\end{remark}

It is important to note that the transition probability matrix $P=P^h+P^v$ is not necessarily a stochastic matrix, that is, the row sum of $P$ may not be equal to one. As $P=B(A^h+A^v)B^{-1}$, the $((i,j),(i',j'))$ entry of $P$ is $p_{(i,j),(i',j')}=b_{i,j}{(A^h+A^v)}_{(i,j),(i',j')}b_{i',j'}^{-1}$. Hence, its row sum is equal to one if and only if $$\sum_{(i',j')\in G^{(2)}_{m-1,n-1}}{(A^h+A^v)}_{(i,j),(i',j')}b_{i',j'}^{-1}=b_{i,j}^{-1},$$ that is, $P$ is a stochastic matrix if and only if one is an eigenvalue of the matrix $A^h+A^v$ with corresponding eigenvector $(b_{i,j}^{-1},\ (i,j)\in G^{(2)}_{m-1,n-1})$. 
To illustrate this, we return to Example \ref{exp1}. In this case, $P$ is a stochastic matrix if and only if one is an eigenvalue of the following matrix:

\begin{equation}\label{AhAv}
		A^h+A^v=\ \begin{array}{c|ccccccccc}
			& (0,0) & (0,1) & (0,2)&(1,0)&(1,1)&(1,2)&(2,0)&(2,1)&(2,2) \\ \hline
			(0,0) & 0 & v_1 & v_3&h_1&0&0&h_3&0&0 \\
			(0,1) & v_1 & 0 & v_2&0&h_1&0 &0&h_3&0 \\
			(0,2) & v_3 & v_2 & 0&0&0&h_1&0&0&h_3\\
			(1,0)&h_1&0&0&0&v_1&v_3&h_2&0&0\\
			(1,1)&0&h_1&0&v_1&0&v_2&0&h_2&0\\
			(1,2)&0&0&h_1&v_3&v_2&0&0&0&h_2\\
			(2,0)&h_3&0&0&h_2&0&0&0&v_1&v_3\\
			(2,1)&0&h_3&0&0&h_2&0&v_1&0&v_2\\
			(2,2)&0&0&h_3&0&0&h_2&v_3&v_2&0\\
		\end{array}
\end{equation}
with corresponding eigenvector $(b_{0,0}^{-1},b_{0,1}^{-1},b_{0,2}^{-1},b_{1,0}^{-1},b_{1,1}^{-1},b_{1,2}^{-1},b_{2,0}^{-1},b_{2,1}^{-1},b_{2,2}^{-1})$.

Note that the matrix (\ref{AhAv}) is irreducible whenever the constants $h_i$ and $v_i$ are positive for all $i=1,2,3$. So, Theorem \ref{pthm2} implies that its largest eigenvalue is positive. Moreover, its corresponding eigenvector has positive entries and it is unique up to a constant multiple $c>0$. Thus, for all $i$, whenever the parameters $h_i$ and $v_i$ are known, these can replaced by $ch_i$ and $cv_i$, respectively, for an appropriate choice of $c>0$. Also, the constants $b_{i,j}$, $(i,j)\in G^{(2)}_{2,2}$ can be chosen which are unique up to a constant multiple such that the row sum of $P$ is equal to one. Further, if the self transition is allowed with constant self transition probability $\alpha\in(0,1)$ then we can replace $h_i$ and $v_i$ with $ch_i$ and $cv_i$ for a suitable constant $c>0$ which allow us to choose $b_{i,j}$ such that the row sum of $P^h+P^v$ equals $1-\alpha$. Hence, the row sum  of transition probability matrix $P=P^h+P^v+\alpha I$ is equal to one, where $I$ is an identity matrix of appropriate order.

\begin{remark}
	Let the GBDP be a step symmetric process in each coordinate directions, that is, $r_{i,j}(x)=l_{i,j}(x)$ and $u_{i,j}(y)=d_{i,j}(y)$ for all $x,y\in\{1,2,\dots,l\}$. Then, the horizontal transition probability matrix $P^h$ reduces to a block diagonal matrix with $n$ identical $m\times m$ symmetric and tri-diagonal block matrices $Q^h$, and $P^v$ reduces to a block diagonal with $m$ identical $n\times n$ symmetric and tri-diagonal matrices $Q^v$. So, the rows and columns of  transition probability matrices can be rearranged as follows:
	\begin{equation*}
		{P^h}_{(i,j),(i',j')}=\begin{cases}
			{Q^h}_{(i,i')},\ j=j',\\
			0,\ \text{otherwise}
		\end{cases}
		\ \ \text{and}\ \  {P^v}_{(i,j),(i',j')}=\begin{cases}
			{Q^v}_{(j,j')},\ i=i',\\
			0,\ \text{otherwise},
		\end{cases}
	\end{equation*}
	where ${P^h}_{(i,j),(i',j')}$ and ${P^v}_{(i,j),(i',j')}$ are the $((i,j),(i',j'))$ entry of $P^h$ and $P^v$, respectively. Also, $Q^h_{(i,i')}$  is an entry of block matrices $Q^h$ corresponding to the row indexed by $(i,j)$ and the column indexed by $(i',j)$ of $P^h$, and  $Q^v_{(j,j')}$ is an entry of $Q^v$ corresponding to the row indexed by $(i,j)$ and  the column indexed by $(i,j')$ of $P^v$.
	
	On using the spectral decomposition theorem for a real symmetric matrix, we get the entries of $(Q^h)^k$ and $(Q^v)^k$ corresponding to $Q^h_{(i,i')}$ and $Q^v_{(j,j')}$ as follows:
	\begin{equation*}
		{(Q^h)^k}_{(i,i')}=\sum_{r=0}^{m-1}(\mu^h_r)^kz^h_r(i)z_r^h(i')\ \ \text{and} \ \ {(Q^v)^k}_{(j,j')}=\sum_{s=0}^{n-1}(\mu^v_s)^kz^v_s(j)z_s^v(j').
	\end{equation*}
	 Here, for $r=0,1,\dots,m-1$ and $s=0,1,\dots,n-1$, $\mu_r^h$  and $\mu_s^v$ are the eigenvalues  of $Q^h$ and $Q^v$ with corresponding orthonormalized eigenvectors $z^h_r$ and $z^v_s$, respectively. Also, $z_r^h(i)$ and $z_s^v(j)$ are the $i$th and $j$th elements of the vectors $z_r^h$ and $z_s^v$, respectively. Thus, the $((i,j),(i',j'))$ entry of the $k$-step transition probability matrix $P^{(k)}$ is given by
	\begin{equation*}
		p^{(k)}_{(i,j),(i',j')}=\sum_{r=0}^{m-1}\sum_{s=0}^{n-1}(\mu_r^h+\mu_s^v)^kz^h_r(i)z_r^h(i')z^v_s(j)z_s^v(j').
	\end{equation*}
\end{remark}

Finally, let us consider a case where $l_1\ne l_2$.
For $m=3$, $n=3$, $l_1=2$ and $l_2=1$ in Example \ref{exp1}, the transition probabilities $l_{2,0}(2)$, $l_{2,1}(2)$, $l_{2,2}(2)$, $d_{0,2}(2)$, $d_{1,2}(2)$ and $d_{2,2}(2)$ are equal to zero. Further, the corresponding matrices $A^h$ and $A^v$ are block diagonal but the block matrices $U^h$ and $U^v$ are not symmetric. Hence, the spectral decomposition method can not be utilized to obtain the $k$-step transition probability matrix. That is, such a decomposition for the transition probability matrix of GBDP on $G^{(2)}_{m-1,n-1}$ is valid if and only if the maximum step size of transition is same for each possible directions. In the next section, we shall see that a similar result holds for the GBDB whose state space is a $q$-dimensional finite grid.

\section{GBDP on $q$-dimensional grid}\label{sec3}	
In this section, we  consider a discrete-time birth-death chain $X=\{X_k\}_{k\ge0}$ whose state space is the following $q$-dimensional finite grid:
\begin{equation*}
G^{(q)}_{n_1,n_2,\dots,n_q}=\{0,1,\dots,n_1\}\times\{0,1,\dots,n_2\}\times\dots\times\{0,1,\dots,n_q\},\ q\ge1.
\end{equation*}
Here, $n_i\ge2$ for all $i=1,2,\dots,q$. So, $G^{(q)}_{n_1,n_2,\dots,n_q}$ is the indexing set for the rows and columns of the transition probability matrix of $X$. Hereafter, we use bold lowercase letters to denote vectors. Let $\textbf{e}_i(x)$, $i\in\{1,2,\dots,q\}$ denote a $q$-dimensional vector whose $i$th coordinate is $x\in\mathbb{N}$ and all other coordinates are $0$. Also, let any two distinct vertices $\textbf{u}$ and $\textbf{v}$ of $G^{(q)}_{n_1,n_2,\dots,n_q}$ which are connected directly to each other be denoted by $\textbf{u}\leftrightarrow\textbf{v}$. So, $\textbf{u}\leftrightarrow\textbf{v}$ if and only if either $\textbf{u}-\textbf{v}\in\{\pm\textbf{e}_i(x)\}$ for some $i\in\{1,2,\dots,q\}$ and $x\in\{1,2,\dots,l_1\}$ or $\textbf{u}-\textbf{v}\in\{\pm\textbf{e}_i(y)\}$ for some $i\in\{1,2,\dots,q\}$ and $y\in\{1,2,\dots,l_2\}$. Further, we assume that $\max\{l_1,l_2\}<\min\{n_1,n_2,\dots,n_q\}$.

 Let $\{\textbf{u},\textbf{v}\}\subset G^{(q)}_{n_1,n_2,\dots,n_q}$ and $P=(p(\textbf{u},\textbf{v}))$ be the one step transition probability matrix of the process $X$. Here, $p(\textbf{u},\textbf{v})=\mathbb{P}\{X_{k+1}=\textbf{v}|X_k=\textbf{u}\}$ denotes the probability of transition from state $\textbf{u}$ to state $\textbf{v}$. Note that $p(\textbf{u},\textbf{v})=0$ whenever $\textbf{u}\nleftrightarrow \textbf{v}$, that is, $\textbf{u}$ and $\textbf{v}$ are not connected by an edge. For $i\in\{1,2,\dots,q\}$, $x\in\{1,2,\dots,l_1\}$ and $y\in\{1,2,\dots,l_2\}$, define the transition probability matrix $P_i=(p_i(\textbf{u},\textbf{v}))$ of $i$th coordinate direction as follows:
\begin{equation*}
	p_i(\textbf{u},\textbf{v})=\begin{cases}
		p(\textbf{u},\textbf{v}),\ \textbf{u}-\textbf{v}=-\textbf{e}_i(x)\ \mathrm{or}\ \textbf{u}-\textbf{v}=\textbf{e}_i(y),\\
		0,\ \mathrm{otherwise}.
	\end{cases}
\end{equation*}
The matrices $P_i$'s are two dimensional analogues of matrices $P^h$ and $P^v$ on $G^{(2)}_{m-1,n-1}$ as defined in Section \ref{sec2}. Similar to the two dimensional case, the matrices $P_i$ and $P_j$ commute with each other if and only if for $(x,x')\in\{1,2,\dots,l_1\}^2$ and $(y,y')\in\{1,2,\dots,l_2\}^2$, the following constraints hold:
	{\small\begin{equation}\label{multequ}
		\begin{aligned}
			p(\textbf{u},\textbf{u}+\textbf{e}_i(x))p(\textbf{u}+\textbf{e}_i(x),\textbf{u}+\textbf{e}_i(x) + \textbf{e}_j(x'))&=p(\textbf{u},\textbf{u}+\textbf{e}_j(x'))p(\textbf{u}+\textbf{e}_j(x'),\textbf{u}+\textbf{e}_j(x')+\textbf{e}_i(x)),\\
			p(\textbf{u},\textbf{u}+\textbf{e}_i(x))p(\textbf{u}+\textbf{e}_i(x),\textbf{u}+\textbf{e}_i(x) - \textbf{e}_j(y))&=p(\textbf{u},\textbf{u}-\textbf{e}_j(y))p(\textbf{u}-\textbf{e}_j(y),\textbf{u}-\textbf{e}_j(y)+\textbf{e}_i(x)),\\
			p(\textbf{u},\textbf{u}-\textbf{e}_i(y))p(\textbf{u}-\textbf{e}_i(y),\textbf{u}-\textbf{e}_i(y) + \textbf{e}_j(x))&=p(\textbf{u},\textbf{u}+\textbf{e}_j(x))p(\textbf{u}+\textbf{e}_j(x),\textbf{u}+\textbf{e}_j(x)-\textbf{e}_i(y))
		\end{aligned}
	\end{equation}}
and
{\small\begin{equation*}
	p(\textbf{u},\textbf{u}-\textbf{e}_i(y))p(\textbf{u}-\textbf{e}_i(y),\textbf{u}-\textbf{e}_i(y) - \textbf{e}_j(y'))=p(\textbf{u},\textbf{u}-\textbf{e}_j(y'))p(\textbf{u}-\textbf{e}_j(y'),\textbf{u}-\textbf{e}_j(y')-\textbf{e}_i(y)).
\end{equation*}}
	
Here, our main goal is to solve the system of equations (\ref{multequ}) for an arbitrary $p(\textbf{u},\textbf{v})$. In the following result, we obtain the solution when all the transition probabilities of $X$ are positive. 	

\begin{theorem}\label{thm2}
	Let $l_1=l_2=l$ and $\{\textbf{u},\textbf{v}\}\subset G^{(q)}_{n_1,n_2,\dots,n_q}$, and let $p(\textbf{u},\textbf{v})>0$ whenever $\textbf{u}\leftrightarrow\textbf{v}$. Then, the transition probability matrices $P_1,P_2,\dots,P_q$ are pairwise commutative if and only if 
	\begin{equation}\label{pf1}
		p(\textbf{u},\textbf{v})=\alpha_{\textbf{u}} \Gamma(\textbf{u},\textbf{v})\alpha_{\textbf{v}}^{-1},
	\end{equation}
	where $\alpha_{\textbf{u}}$ and $\alpha_{\textbf{v}}$ are some positive constants which are unique up to constant multiplication. Here, $\Gamma(\textbf{u},\textbf{v})$ are unique positive constants such that $\Gamma(\textbf{u},\textbf{v})=\Gamma(\textbf{u}+\textbf{u}',\textbf{v}+\textbf{u}')$ whenever $\textbf{u}-\textbf{v}=e_i(x)$ or $\textbf{u}-\textbf{v}=-e_i(x)$, $x\in\{1,2,\dots,l\}$, $i=1,2,\dots,q$ and for some $\textbf{u}'\in\{\sum_{j\ne i}a_{j}\textbf{e}_j(1):a_{j}\in\mathbb{Z}\}$.
\end{theorem}
	
\begin{proof}
	For if part, let us assume that (\ref{pf1}) holds. Then, we have
	\begin{align*}			p(\textbf{u},\textbf{u}+&\textbf{e}_i(x))p(\textbf{u}+\textbf{e}_i(x),\textbf{u}+\textbf{e}_i(x) + \textbf{e}_j(x'))\\
		&=\alpha_{\textbf{u}}\Gamma(\textbf{u},\textbf{u}+\textbf{e}_i(x))\alpha_{\textbf{u}+\textbf{e}_i(x)}^{-1}\alpha_{\textbf{u}+\textbf{e}_i(x)}\Gamma(\textbf{u}+\textbf{e}_i(x),\textbf{u}+\textbf{e}_i(x) + \textbf{e}_j(x'))\alpha_{\textbf{u}+\textbf{e}_i(x) + \textbf{e}_j(x')}^{-1}\\
		&=\alpha_{\textbf{u}}\Gamma(\textbf{u},\textbf{u}+\textbf{e}_j(x'))\alpha_{\textbf{u}+\textbf{e}_j(x')}^{-1}\alpha_{\textbf{u}+\textbf{e}_j(x')}\Gamma(\textbf{u}+\textbf{e}_j(x'),\textbf{u}+\textbf{e}_j(x') + \textbf{e}_i(x))\alpha_{\textbf{u}+\textbf{e}_j(x') + \textbf{e}_i(x)}^{-1}\\
		&=p(\textbf{u},\textbf{u}+\textbf{e}_j(x'))p(\textbf{u}+\textbf{e}_j(x'),\textbf{u}+\textbf{e}_j(x')+\textbf{e}_i(x)),
	\end{align*}
	where we have used $\Gamma(\textbf{u},\textbf{u}+\textbf{e}_i(x))=\Gamma(\textbf{u}+\textbf{e}_j(x'),\textbf{u}+\textbf{e}_j(x') + \textbf{e}_i(x))$ and $\Gamma(\textbf{u}+\textbf{e}_i(x),\textbf{u}+\textbf{e}_i(x) + \textbf{e}_j(x'))=\Gamma(\textbf{u},\textbf{u}+\textbf{e}_j(x'))$ to get the second equality. This prove that the first equality in (\ref{multequ}) holds.  Similarly, the remaining three equalities in (\ref{multequ}) can be established. Thus, the transition matrices $P_i$'s are pairwise commutative.
	
	 For only if part, we show that if $P_1,P_2,\dots,P_q$ commute then $X$ is a reversible chain. In view of Theorem \ref{pthm3}, it is enough to show that there exist a measure $\boldsymbol{\beta}=(\beta_\textbf{u}:\textbf{u}\in G^{(q)}_{n_1,n_2,\dots,n_q})$ such that $\beta_\textbf{u} p(\textbf{u},\textbf{v})=\beta_{\textbf{v}} p(\textbf{v},\textbf{u})$ for all $\{\textbf{u},\,\textbf{v}\}\subset G^{(q)}_{n_1,n_2,\dots,n_q}$.
	 
	 To show the existence of measure $\boldsymbol{\beta}$, let $\textbf{0}$ be the zero vector in $G^{(q)}_{n_1,n_2,\dots,n_q}$ and denote $||\textbf{u}||=u_1+u_2+\dots+u_q$ for $\textbf{u}\in G^{(q)}_{n_1,n_2,\dots,n_q}$. Now, for any fix positive constant $\beta_{\textbf{0}}$, we construct a measure $\boldsymbol{\beta}$ recursively. For $||\textbf{u}||\in\{1,2,\dots,l\}$ such that $\textbf{u}=\textbf{e}_i(x)$, $x\in\{1,2,\dots,l\}$ for some $i$, we take $\beta_{\textbf{u}}=\beta_{\textbf{0}}p(\textbf{0},\textbf{u})/p(\textbf{u},\textbf{0})$. Also, for $||\textbf{u}||=1$, $\textbf{u}\in G^{(q)}_{n_1,n_2,\dots,n_q}$, we need to find a $\textbf{v}\in G^{(q)}_{n_1,n_2,\dots,n_q}$ with $||\textbf{v}||\in\{2,3,\dots,l+1\}$ such that $\beta_{\textbf{u}}p(\textbf{u},\textbf{v})=\beta_{\textbf{v}}p(\textbf{v},\textbf{u})$. Equivalently, we need to find $\beta_{\textbf{v}}$, $\textbf{v}\in G^{(q)}_{n_1,n_2,\dots,n_q}$ which satisfy $\beta_{\textbf{u}}p(\textbf{u},\textbf{v})=\beta_{\textbf{v}}p(\textbf{v},\textbf{u})$ for all $\textbf{u}=\textbf{v}-\textbf{e}_i(x)$, $x\in\{1,2,\dots,||\textbf{v}||-1\}$ with $||\textbf{u}||=1$ for some $i\in\{1,2,\dots,q\}$. 
	 
	 We have the following two cases:\\ 
	 \textbf{\textit{Case I.}} If $\textbf{v}=x\textbf{e}_i(1)$, $x\in\{2,3,\dots,l+1\}$ for some $i\in\{1,2,\dots,q\}$ then we take
	 \begin{equation*}
	 	\beta_{\textbf{v}}=\beta_{\textbf{e}_i(1)}\frac{p(\textbf{e}_i(1),x\textbf{e}_i(1))}{p(x\textbf{e}_i(1),\textbf{e}_i(1))}=\beta_{\textbf{0}}\frac{p(\textbf{0},\textbf{e}_i(1))p(\textbf{e}_i(1),x\textbf{e}_i(1))}{p(\textbf{e}_i(1),\textbf{0})p(x\textbf{e}_i(1),\textbf{e}_i(1))}.
	 \end{equation*}\\
	 \textbf{\textit{Case II.}} If $\textbf{v}=\textbf{e}_j(1)+\textbf{e}_k(x)$, $x\in\{1,2,\dots,l\}$ for some $j\ne k$ then for $x=1$, we take 
	 \begin{equation*}
	 	\beta_{\textbf{e}_j(1)+\textbf{e}_k(1)}=\beta_{\textbf{e}_j(1)}\frac{p(\textbf{e}_j(1),\textbf{e}_j(1)+\textbf{e}_k(1))}{p(\textbf{e}_j(1)+\textbf{e}_k(1),\textbf{e}_j(1))}=\beta_{\textbf{0}}\frac{p(\textbf{0},\textbf{e}_j(1))p(\textbf{e}_j(1),\textbf{e}_j(1)+\textbf{e}_k(1))}{p(\textbf{e}_j(1),\textbf{0})p(\textbf{e}_j(1)+\textbf{e}_k(1),\textbf{e}_j(1))}
	 \end{equation*}
	 or
	 \begin{equation*}
	 	\beta_{\textbf{e}_j(1)+\textbf{e}_k(1)}=\beta_{\textbf{e}_k(1)}\frac{p(\textbf{e}_k(1),\textbf{e}_k(1)+\textbf{e}_j(1))}{p(\textbf{e}_k(1)+\textbf{e}_j(1),\textbf{e}_k(1))}=\beta_{\textbf{0}}\frac{p(\textbf{0},\textbf{e}_k(1))p(\textbf{e}_k(1),\textbf{e}_k(1)+\textbf{e}_j(1))}{p(\textbf{e}_k(1),\textbf{0})p(\textbf{e}_k(1)+\textbf{e}_j(1),\textbf{e}_k(1))}.
	 \end{equation*} 
	 However, on using (\ref{multequ}), we have
	 \begin{equation*}
	 	p(\textbf{0},\textbf{e}_k(1))p(\textbf{e}_k(1),\textbf{e}_k(1)+\textbf{e}_j(1))=p(\textbf{0},\textbf{e}_j(1))p(\textbf{e}_j(1),\textbf{e}_j(1)+\textbf{e}_k(1))
	 \end{equation*}
	 and
	 \begin{equation*}
	 	p(\textbf{e}_j(1),\textbf{0})p(\textbf{e}_j(1)+\textbf{e}_k(1),\textbf{e}_j(1))=p(\textbf{e}_k(1),\textbf{0})p(\textbf{e}_k(1)+\textbf{e}_j(1),\textbf{e}_k(1)).
	 \end{equation*}
	 Thus, $\beta_{\textbf{e}_j(1)+\textbf{e}_k(1)}$ is uniquely determined. Now, suppose $1<x\leq l$. Then, $\beta_{\textbf{v}}$ must have the following unique representation:
	 \begin{equation*}
	 	\beta_{\textbf{v}}=\beta_{\textbf{e}_j(1)+\textbf{e}_k(x)}=\beta_{\textbf{e}_j(1)}\frac{p(\textbf{e}_j(1),\textbf{e}_j(1)+\textbf{e}_k(x))}{p(\textbf{e}_j(1)+\textbf{e}_k(x),\textbf{e}_j(1))}=\beta_{\textbf{0}}\frac{p(\textbf{0},\textbf{e}_j(1))p(\textbf{e}_j(1),\textbf{e}_j(1)+\textbf{e}_k(x))}{p(\textbf{e}_j(1),\textbf{0})p(\textbf{e}_j(1)+\textbf{e}_k(x),\textbf{e}_j(1))}.
	 \end{equation*} 
	 Proceeding along the similar lines, for any $\textbf{u}\in G^{(q)}_{n_1,n_2,\dots,n_q}$, we choose $\textbf{u}_0=\textbf{0}$, $\textbf{u}_1$, $\dots$, $\textbf{u}_N=\textbf{u}$ from $G^{(q)}_{n_1,n_2,\dots,n_q}$ such that $\textbf{u}_{r+1}-\textbf{u}_r\in\{\textbf{e}_1(x),\textbf{e}_2(x),\dots,\textbf{e}_q(x)\}_{1\leq x\leq l}$, $0\leq r\leq N-1$, $N\in\mathbb{N}$ and take
	 \begin{equation}\label{p3}
	 	\beta_{\textbf{u}}=\prod_{r=0}^{N-1}\frac{p(\textbf{u}_r,\textbf{u}_{r+1})}{p(\textbf{u}_{r+1},\textbf{u}_r)}. 
	 \end{equation}
	 Note that the right hand side of (\ref{p3}) is independent of the choice of $\textbf{u}_1,\textbf{u}_2,\dots,\textbf{u}_{N-1}$. Moreover, given the value of $\beta_\textbf{0}$, the constant $\beta_{\textbf{u}}$, $\textbf{u}\in G^{(q)}_{n_1,n_2,\dots,n_q}$ are uniquely determined by (\ref{p3}) that satisfy the property $\beta_{\textbf{u}}p(\textbf{u},\textbf{v})=\beta_{\textbf{v}}p(\textbf{v},\textbf{u})$ for all $\{\textbf{u},\textbf{v}\}\subset G^{(q)}_{n_1,n_2,\dots,n_q}$. This prove that $X$ is reversible.
	 
	  Now, we take $\alpha_{\textbf{u}}=\beta_{\textbf{u}}^{-1/2}$ and define $\Gamma(\textbf{u},\textbf{v})$ as follows:
	 \begin{equation}\label{Gamma1}
	 	\Gamma(\textbf{u},\textbf{v})=\alpha_{\textbf{u}}^{-1}p(\textbf{u},\textbf{v})\alpha_{\textbf{v}}.
	 \end{equation}
	 Then, 
	 \begin{equation}\label{Gamma2}
	 	p(\textbf{u},\textbf{v})=\alpha_{\textbf{u}}\Gamma(\textbf{u},\textbf{v})\alpha_{\textbf{v}}^{-1}\ \ \mathrm{and}\ \ \Gamma(\textbf{u},\textbf{v})=\Gamma(\textbf{v},\textbf{u}),
	 \end{equation}
	 where we have used $\beta_{\textbf{u}}p(\textbf{u},\textbf{v})=\beta_{\textbf{v}}p(\textbf{v},\textbf{u})$ to get the second equality of (\ref{Gamma2}).
	 Conversely, if there exist some constants $\alpha_{\textbf{u}}$ such that  (\ref{Gamma1}) and (\ref{Gamma2}) hold then $\alpha_{\textbf{u}}^{-2}p(\textbf{u},\textbf{v})=\alpha_{\textbf{v}}^{-2}p(\textbf{v},\textbf{u})$ for all $\{\textbf{u},\textbf{v}\}\subset G^{(q)}_{n_1,n_2,\dots,n_q}$, that is, $X$ is a reversible Markov chain. Thus, we proved that there exist constants $\Gamma(\textbf{u},\textbf{v})$ and $\alpha_{\textbf{u}}$, $\textbf{u}\in G^{(q)}_{n_1,n_2,\dots,n_q}$ which agree with (\ref{Gamma1}) and (\ref{Gamma2}). 
	 
	 Let $\{\tilde{\alpha}_{\textbf{u}}$, $\textbf{u}\in G^{(q)}_{n_1,n_2,\dots,n_q}\}$ be another set of constants that satisfies the conditions (\ref{Gamma1}) and (\ref{Gamma2}). Then, $\alpha_{\textbf{u}}\alpha_{\textbf{v}}^{-1}=\tilde{\alpha}_{\textbf{u}}\tilde{\alpha}_{\textbf{v}}^{-1}$ for all $\{\textbf{u},\,\textbf{v}\}\subset G^{(q)}_{n_1,n_2,\dots,n_q}$. Thus, $\alpha_{\textbf{u}}$'s are unique up to constant multiplication and $\Gamma(\textbf{u},\textbf{v})$'s are uniquely determined.
	 
	  As $p(\textbf{u},\textbf{v})$ solves the system of equations given in (\ref{multequ}) so does $\Gamma(\textbf{u},\textbf{v})$. Thus, the set of constraints corresponding to a particular sub-graph of $G^{(q)}_{n_1,n_2,\dots,n_q}$ with vertices $\{\textbf{u},\textbf{u}+\textbf{e}_i(x),\textbf{u}+\textbf{e}_j(x),\textbf{u}+\textbf{e}_i(y),\textbf{u}+\textbf{e}_j(y),\textbf{u}+\textbf{e}_i(x)+\textbf{e}_j(x'),\textbf{u}+\textbf{e}_i(x)+\textbf{e}_j(y),\textbf{u}+\textbf{e}_i(y)+\textbf{e}_j(x),\textbf{u}+\textbf{e}_i(y)+\textbf{e}_j(y'):\ \{x,x',y,y'\}\subset\{1,2,\dots,l\}\}$ is given by
	 {\small\begin{equation}\label{multequgamma}
	 	\begin{aligned}
	 		&\Gamma(\textbf{u},\textbf{u}+\textbf{e}_i(x))\Gamma(\textbf{u}+\textbf{e}_i(x),\textbf{u}+\textbf{e}_i(x) + \textbf{e}_j(x'))
	 		=\Gamma(\textbf{u},\textbf{u}+\textbf{e}_j(x'))\Gamma(\textbf{u}+\textbf{e}_j(x'),\textbf{u}+\textbf{e}_j(x')+\textbf{e}_i(x)),\\
	 		&\Gamma(\textbf{u}+\textbf{e}_i(y),\textbf{u}+\textbf{e}_i(y)+\textbf{e}_j(x))\Gamma(\textbf{u}+\textbf{e}_i(y)+\textbf{e}_j(x),\textbf{u}+\textbf{e}_j(x))
	 		=\Gamma(\textbf{u}+\textbf{e}_i(y),\textbf{u})\Gamma(\textbf{u},\textbf{u}+\textbf{e}_j(x)),\\
	 		&\Gamma(\textbf{u}+\textbf{e}_i(y)+\textbf{e}_j(y'),\textbf{u}+\textbf{e}_j(y'))\Gamma(\textbf{u}+ \textbf{e}_j(y'),\textbf{u})=\Gamma(\textbf{u}+\textbf{e}_i(y)+\textbf{e}_j(y'),\textbf{u}+\textbf{e}_i(y))\Gamma(\textbf{u}+\textbf{e}_i(y),\textbf{u})
	 	\end{aligned}
	 \end{equation}	}
	 and
	 {\small\begin{equation*}
	 	\Gamma(\textbf{u}+\textbf{e}_j(y),\textbf{u})\Gamma(\textbf{u},\textbf{u}+\textbf{e}_i(x))=\Gamma(\textbf{u}+\textbf{e}_j(y),\textbf{u}+\textbf{e}_j(y)+\textbf{e}_i(x))\Gamma(\textbf{u}+\textbf{e}_j(y)+\textbf{e}_i(x),\textbf{u}+\textbf{e}_i(x)).
	 \end{equation*}}
	  As $x$, $x'$, $y$ and $y'$ take values in same set and (\ref{multequgamma}) holds for their all possible values, without loss of generality, we can take $y'=x$ and $y=x'$ in (\ref{multequgamma}). On using the second equality in (\ref{Gamma2}), it can be shown that the set of four equalities in (\ref{multequgamma}) is equivalent to its first two equalities, that is, in presence of first two equalities the last two equalities  are redundant. Hence, the system of equations (\ref{multequgamma}) is equivalent to 
	  \begin{equation*}
	  	\Gamma(\textbf{u},\textbf{u}+\textbf{e}_i(x))\Gamma(\textbf{u}+\textbf{e}_i(x),\textbf{u}+\textbf{e}_i(x) + \textbf{e}_j(y))
	  	=\Gamma(\textbf{u},\textbf{u}+\textbf{e}_j(y))\Gamma(\textbf{u}+\textbf{e}_j(y),\textbf{u}+\textbf{e}_j(y)+\textbf{e}_i(x))
	  	\end{equation*}
	  	and
	  	\begin{equation*}
	  	\Gamma(\textbf{u}+\textbf{e}_i(y),\textbf{u}+\textbf{e}_i(y)+\textbf{e}_j(x))\Gamma(\textbf{u}+\textbf{e}_j(x),\textbf{u}+\textbf{e}_i(y)+\textbf{e}_j(x))=\Gamma(\textbf{u},\textbf{u}+\textbf{e}_i(y))\Gamma(\textbf{u},\textbf{u}+\textbf{e}_j(x)).
	  \end{equation*}
	  This system further reduces to
	\begin{equation*}
			\Gamma(\textbf{u},\textbf{u}+\textbf{e}_i(x))=\Gamma(\textbf{u}+\textbf{e}_j(y),\textbf{u}+\textbf{e}_j(y)+\textbf{e}_i(x))
			\end{equation*}
			and
			\begin{equation*}
			\Gamma(\textbf{u},\textbf{u}+\textbf{e}_j(y))=\Gamma(\textbf{u}+\textbf{e}_i(x),\textbf{u}+\textbf{e}_i(x)+\textbf{e}_j(y)).
	\end{equation*}	
On iterating these two equations, we get $\Gamma(\textbf{u},\textbf{u}+\textbf{e}_i(x))=\Gamma(\textbf{u}+a_{j}\textbf{e}_j(y),\textbf{u}+a_{j}\textbf{e}_j(y)+\textbf{e}_i(x))$ for some $a_{j}\in\mathbb{Z}$. Iterating further will establish the additional properties of $\Gamma(\textbf{u},\textbf{v})$ as stated in the theorem. This completes the proof.
\end{proof}
\begin{remark}
	For $l_1=l_2=1$, Theorem \ref{thm2} reduces to Theorem 3.1 of Evans \textit{et al.} (2010).
\end{remark}
\begin{remark}\label{3.1}
	Note that if $\textbf{u}-\textbf{v}\in\{\pm\textbf{e}_i(1), \pm\textbf{e}_i(2),\dots, \pm\textbf{e}_i(l)\}$ then there exist a unique $r\in\{0,1,2,\dots,n_i-x\}$, $i\in\{1,2,\dots,q\}$ and $\textbf{w}\in\{\sum_{j\ne i}a_{j}\textbf{e}_j(1):a_j\in\mathbb{Z}\}$ such that $\textbf{u}=r\textbf{e}_i(1)+\textbf{w}$ and $\textbf{v}=(r+x)\textbf{e}_i(1)+\textbf{w}$. Therefore, in view of Theorem \ref{thm2}, we have 
	\begin{equation}\label{form}
		\Gamma(\textbf{u},\textbf{v})=\Gamma(r\textbf{e}_i(1),(r+x)\textbf{e}_i(1))
	\end{equation}
	for all $\{\textbf{u},\,\textbf{v}\}\subset G^{(q)}_{n_1,n_2,\dots,n_q}$ and $x\in\{1,2,\dots,l\}$. Thus, the parametrization (\ref{pf1}) involves exactly $\sum_{i=1}^{q}\sum_{x=1}^{l}(n_i-x+1)$ many unique parameters of the form (\ref{form}). Moreover, it involves exactly $(n_1+1)(n_2+1)\cdots (n_q+1)$ many parameters $\alpha_{\textbf{u}}$, $\textbf{u}\in G^{(q)}_{n_1,n_2,\dots,n_q}$ which are unique up to constant multiplication.
	
	Usually, $\alpha_{\textbf{u}}$ is referred as the vertex parameter and $\Gamma(\textbf{u},\textbf{v})$ is referred as the edge parameter whenever $\textbf{u}$ and $\textbf{v}$ are connected by an edge.
\end{remark}
\begin{example}
	For $q=2$, $n_1=2$, $n_2=2$ and $l=2$, the parametrization (\ref{pf1}) reduces to the parametrization given in (\ref{matricesrank}). Moreover, 
	$h_1=\Gamma((0,0),(1,0))$, $h_2=\Gamma((1,0),(2,0))$, $h_3=\Gamma((0,0),(2,0))$, $v_1=\Gamma((0,0),(0,1))$, $v_2=\Gamma((0,1),(0,2))$, $v_3=\Gamma((0,0),(0,2))$ and
	$\alpha_{(i,j)}=b_{i,j}$ for all $(i,j)\in \{(0,0),(0,1),(0,2),(1,0),(1,1),(1,2),(2,0),(2,1),(2,2)\}$.
\end{example}

Note that for $l_1=l_2$, if $p(\textbf{u},\textbf{v})>0$ whenever $\textbf{u}\leftrightarrow\textbf{v}$ then there are some redundancies in the system of equations (\ref{multequ}), that is, any three of the equalities given in (\ref{multequ}) deduce the fourth one. 
Also, all the equations given in (\ref{multequ}) are of the form $p(\textbf{u}_1,\textbf{u}_2)p(\textbf{u}_2,\textbf{v}_1)-p(\textbf{u}_1,\textbf{v}_2)p(\textbf{v}_2,\textbf{v}_1)=0$ for all $\{\textbf{u}_1,\textbf{u}_2,\textbf{v}_1,\textbf{v}_2\}\subset G^{(q)}_{n_1,n_2,\dots,n_q}$ such that $\textbf{u}_1\leftrightarrow\textbf{u}_2$, $\textbf{u}_2\leftrightarrow\textbf{v}_1$, $\textbf{u}_1\leftrightarrow\textbf{v}_2$ and $\textbf{v}_2\leftrightarrow\textbf{v}_1$. Equivalently, these equations are of the form $p^*(\textbf{u}_1,\textbf{u}_2)+p^*(\textbf{u}_2,\textbf{v}_1)-p^*(\textbf{u}_1,\textbf{v}_2)-p^*(\textbf{v}_2,\textbf{v}_1)=0$ where $p^*(\textbf{u},\textbf{v})=\log p(\textbf{u},\textbf{v})$.

Next, for $l_1=l_2=l$, we identify the constraints that are independent. First, we define two matrices and fix some notations.

\subsection{Notations}Let us consider a matrix $\mathcal{Q}_{n_1,n_2,\dots,n_q}^{l}$ that has one row for each constraint in (\ref{multequ}) and whose columns are indexed by pairs $(\textbf{u},\textbf{v})\in{G}^{(q)}\times{G}^{(q)}$ whenever $\textbf{u}\leftrightarrow\textbf{v}$. Here, the row of $\mathcal{Q}_{n_1,n_2,\dots,n_q}^l$ corresponding to the constraint $p(\textbf{u}_1,\textbf{u}_2)p(\textbf{u}_2,\textbf{v}_1)=p(\textbf{u}_1,\textbf{v}_2)p(\textbf{v}_2,\textbf{v}_1)$ has entries $1$, $1$, $-1$, $-1$ in columns corresponding to the pairs $(\textbf{u}_1,\textbf{u}_2)$, $(\textbf{u}_2,\textbf{v}_1)$, $(\textbf{u}_1,\textbf{v}_2)$, $(\textbf{v}_2,\textbf{v}_1)$, respectively, and zero elsewhere.

Also, consider a matrix $\mathcal{R}_{n_1,n_2,\dots,n_q}^{l}$ which has a column for each order pairs $(\textbf{u},\textbf{v})\in G^{(q)}_{n_1,n_2,\dots,n_q}\times G^{(q)}_{n_1,n_2,\dots,n_q}$ whenever $\textbf{u}\leftrightarrow\textbf{v}$ and it has a row corresponding to each parameter mentioned in Remark \ref{3.1}. In $\mathcal{R}_{n_1,n_2,\dots,n_q}^{l}$, the column associated with the pair $(\textbf{u},\textbf{v})$ has entries $1$, $-1$, $1$ corresponding to the parameters $\alpha_{\textbf{u}}$, $\alpha_{\textbf{v}}$, $\Gamma(\textbf{u},\textbf{v})$, respectively, and zero elsewhere.
\begin{remark}
	The orders of matrices $\mathcal{Q}^l_{n_1,n_2,\dots,n_q}$ and $\mathcal{R}^l_{n_1,n_2,\dots,n_q}$ are
	{\small\begin{equation*}\label{OQ}
		\bigg(4\sum_{1\leq j<i\leq q}\sum_{x=1}^{l}\sum_{y=1}^{l}(n_j-y+1)(n_i-x+1)\prod_{k\ne i,j}(n_k+1)\bigg)\times\bigg(2\sum_{i=1}^{q}\sum_{x=1}^{l}(n_i-x+1)\prod_{j\ne i}(n_j+1)\bigg)
	\end{equation*}}
	and
	{\small\begin{equation*}\label{OR}
		\bigg(\sum_{i=1}^{q}\sum_{x=1}^{l}(n_i-x+1)+\prod_{i=1}^{q}(n_i+1)\bigg)\times\bigg(2\sum_{i=1}^{q}\sum_{x=1}^{l}(n_i-x+1)\prod_{j\ne i}(n_j+1)\bigg),
	\end{equation*}}
	respectively. In particular, the order of matrix $\mathcal{Q}^1_{n_1,n_2,\dots,n_q}$ is
	{\small\begin{equation*}\label{f1}
		\bigg(4\sum_{1\leq j<i\leq q}n_jn_i\prod_{k\neq i,j}(n_k+1)\bigg)\times \bigg(2\sum_{i=1}^{q}n_i\prod_{j\ne i}(n_j+1)\bigg)
	\end{equation*}}
	and the order of $\mathcal{R}^1_{n_1,n_2,\dots,n_q}$ is given by
	{\small\begin{equation*}\label{ff2}
		\bigg(\sum_{i=1}^{q}n_i+\prod_{i=1}^{q}(n_i+1)\bigg)\times\bigg(2\sum_{i=1}^{q}n_i\prod_{j\ne i}(n_j+1)\bigg).
	\end{equation*}}
	This agrees with the results obtained in Remark 3.5 and Remark 3.8 of Evans \textit{et al.} (2010), respectively.
\end{remark}
	
\begin{example}
	For $q=2$, $l=2$, $n_1=2$ and $n_2=2$, the orders of matrices $\mathcal{Q}_{2,2}^2$ and $\mathcal{R}_{2,2}^2$ are $36\times 36$ and $15\times36$, respectively. Let $E$ be the collection of all ordered pairs $(\textbf{u},\textbf{v})\in
		\{(0,0),(0,1),(0,2),(1,0),(1,1),(1,2),(2,0), (2,1), (2,2)\}^2
$ such that $\textbf{u}\leftrightarrow\textbf{v}$. Then, the cardinality of $E$ is $36$. 

 Let $\mathscr{R}=(r_t,\,t\in E)$ and $\mathscr{R}'=(r'_t,\,t\in E)$ be the rows of $\mathcal{Q}_{2,2}^2$ associated with the constraints  $r_{0,0}(1)u_{1,0}(1)=u_{0,0}(1)r_{0,1}(1)$ and $u_{1,0}(1)l_{1,1}(1)=l_{1,0}(1)u_{0,0}(1)$, respectively. Then,
	\begin{equation*}
		r_t=\begin{cases}
			1,\ t\in\{(0,0)\times(1,0),(1,0)\times(1,1)\},\\
			-1,\ t\in\{(0,0)\times(0,1),(0,1)\times(1,1)\},\\
			0,\ \mathrm{otherwise}
		\end{cases}
	\end{equation*}
	and
	\begin{equation*}
		r'_t=\begin{cases}
			1,\ t\in\{(1,0)\times(1,1),(1,1)\times(0,1)\},\\
			-1,\ t\in\{(1,0)\times(0,0),(0,0)\times(0,1)\},\\
			0,\ \mathrm{otherwise}.
		\end{cases}
	\end{equation*}
	Thus, the dot product of vectors $\mathscr{R}$ and $\mathscr{R}'$ is equal to $2$, that is, $\mathscr{R}\cdot\mathscr{R}'=2$. Now, let us consider another row $\mathscr{R}{''}=(r_t{''},\,t\in E)$ of $\mathcal{Q}_{2,2}^2$ associated with the constraint $l_{2,1}(2)d_{0,1}(1)=l_{2,0}(2)d_{2,1}(1)$. Then,
	\begin{equation*}
		r_t{''}=\begin{cases}
			1,\ t\in\{(2,1)\times(0,1), (0,1)\times(0,0)\},\\
			-1,\ t\in\{(2,0)\times(0,0), (2,1)\times(2,0)\},\\
			0,\ \mathrm{otherwise}.
		\end{cases}
	\end{equation*}
	Here, $\mathscr{R}{''}\cdot\mathscr{R}=0$ and $\mathscr{R}{''}\cdot\mathscr{R}'=0$. Therefore, we conclude that some rows of $\mathcal{Q}_{2,2}^2$ are orthogonal to each other but not all.
	
	 Let $\mathscr{R}_{\textbf{u}}=(r_t^{\textbf{u}},\,t\in E)$ be the row of $\mathcal{R}_{2,2}^2$ associated with vertex parameter $\alpha_{\textbf{u}}$ for all $\textbf{u}\in\{(0,0),(0,1),(0,2),(1,0)$, $(1,1),(1,2),(2,0), (2,1), (2,2)\}$. Then, we have
	\begin{align*}
		r_t^{(0,0)}&=\begin{cases}
			1,\ t\in\{(0,0)\times(1,0), (0,0)\times(0,1), (0,0)\times(2,0), (0,0)\times(0,2)\},\\
			-1,\ t\in\{(1,0)\times(0,0), (0,1)\times(0,0), (2,0)\times(0,0), (0,2)\times(0,0)\},\\
			0,\ \mathrm{otherwise},
		\end{cases}\\
		r_t^{(0,1)}&=\begin{cases}
			1,\ t\in\{(0,1)\times(0,0), (0,1)\times(1,1), (0,1)\times(2,1), (0,1)\times(0,2)\},\\
			-1,\ t\in\{(0,0)\times(0,1), (1,1)\times(0,1), (2,1)\times(0,1), (0,2)\times(0,1)\},\\
			0,\ \mathrm{otherwise},
		\end{cases}
	\end{align*}
and
\begin{equation*}
	r_t^{(2,2)}=\begin{cases}
		1,\ t\in\{(2,2)\times(2,0), (2,2)\times(0,2), (2,2)\times(1,2), (2,2)\times(2,1)\},\\
		-1,\ t\in\{(2,0)\times(2,2), (0,2)\times(2,2), (1,2)\times(2,2), (2,1)\times(2,2)\},\\
		0,\ \mathrm{otherwise}.
	\end{cases}
\end{equation*}
Here, $\mathscr{R}_{(0,0)}\cdot\mathscr{R}_{(0,1)}=-2$, $\mathscr{R}_{(0,0)}\cdot\mathscr{R}_{(2,2)}=0$ and $\mathscr{R}_{(0,1)}\cdot\mathscr{R}_{(2,2)}=0$. Note that $\mathscr{R}_{\textbf{u}}\cdot\mathscr{R}_{\textbf{v}}=0$ if $\textbf{u}\nleftrightarrow\textbf{v}$ and $\mathscr{R}_{\textbf{u}}\cdot\mathscr{R}_{\textbf{v}}=-2$ if $\textbf{u}\leftrightarrow\textbf{v}$. Also, the dot products of $\mathscr{R}_{(0,0)}$, $\mathscr{R}_{(0,1)}$ and $\mathscr{R}_{(2,2)}$ with the vectors in span of the rows of $\mathcal{Q}^{(2)}_{2,2}$ are all $0$. In Proposition \ref{thm3}, we shall show that this result is true for any finite dimensional finite grid.
\end{example}

Next, we aim to determine the maximal number of independent constraints required for the commutation of transition probability matrices when $p(\textbf{u},\textbf{v})>0$. To achieve this, we determine the dimension of vector space spanned by the rows of $\mathcal{Q}^l_{n_1,n_2,\dots,n_q}$. That is, we need to determine the rank of matrix  $\mathcal{Q}^l_{n_1,n_2,\dots,n_q}$.
	
Let $\mathcal{V}_{\mathcal{Q}}$ and $\mathcal{V}_{\mathcal{R}}$ be the vector spaces spanned by the rows of matrices $\mathcal{Q}^l_{n_1,n_2,\dots,n_q}$ and $\mathcal{R}^l_{n_1,n_2,\dots,n_q}$, respectively. Then, we have the following result that will be useful to determine the maximal number of linearly independent rows of the constraint matrix $\mathcal{Q}^l_{n_1,n_2,\dots,n_q}$.
Its proof follows similar lines to the case of $l=1$ (see Corollary 3.10, Evans \textit{et al.} (2010)). Thus, it is omitted.
\begin{proposition}\label{thm3}
	The vector spaces $\mathcal{V}_{\mathcal{Q}}$ and $\mathcal{V}_{\mathcal{R}}$ are orthogonal complements of each other. 
\end{proposition}

The following lemma gives the rank of matrix $\mathcal{R}^l_{n_1,n_2,\dots,n_q}$ that together with Proposition \ref{thm3} determine the rank of matrix $\mathcal{Q}^l_{n_1,n_2,\dots,n_q}$.
\begin{lemma}\label{lemma1}
	The rank of matrix $\mathcal{R}^l_{n_1,n_2,\dots,n_q}$ is given by
	\begin{equation}\label{r1}
		\mathrm{rank}(\mathcal{R}^l_{n_1,n_2,\dots,n_q})=l\sum_{i=1}^{q}n_i+\prod_{i=1}^{q}(n_i+1)-\frac{q(l-1)l}{2}-1.
	\end{equation}
\end{lemma}
\begin{proof}
	In view of Remark \ref{prem1}, it is sufficient to obtain the rank of matrix $\mathcal{R}^l_{n_1,n_2,\dots,n_q}{\mathcal{R}^l_{n_1,n_2,\dots,n_q}}^T$. Let $R_{\textbf{u}}$ be the row vector of $\mathcal{R}^l_{n_1,n_2,\dots,n_q}$ corresponding to the parameter $\alpha_{\textbf{u}}$, $\textbf{u}\in G^{(q)}_{n_1,n_2,\dots,n_q}$. The vertex parameter $\alpha_{\textbf{u}}$ appears in the transition probabilities $p(\textbf{u},\textbf{v})$ and $p(\textbf{v},\textbf{u})$ whenever $\textbf{u}\leftrightarrow\textbf{v}$. So, the dot product $R_{\textbf{u}}\cdot R_{\textbf{u}}$ is equal to the total number of all such vertex $\textbf{v}\in G^{(q)}_{n_1,n_2,\dots,n_q}$, that is, it is equal to twice of the degree of vertex $\textbf{u}$. Moreover, $R_{\textbf{u}}\cdot R_{\textbf{v}}=0$ if $\textbf{u}\nleftrightarrow\textbf{v}$ and  $R_{\textbf{u}}\cdot R_{\textbf{v}}=-2$ if $\textbf{u}\leftrightarrow\textbf{v}$.
	
	Let $R_i^{r,x}$ denote the row vector of  $\mathcal{R}^l_{n_1,n_2,\dots,n_q}$ corresponding to the edge parameter $\Gamma(r\textbf{e}_i(1),(r+x)\textbf{e}_i(1))$, where $r\in\{0,1,\dots,n_i-x\}$ and $x\in\{1,2,\dots,l\}$. Note that the parameter $\Gamma(r\textbf{e}_i(1),(r+x)\textbf{e}_i(1))$, $i\in\{1,2,\dots,q\}$ appears in two transition probabilities $p(r\textbf{e}_i(1)+w,(r+x)\textbf{e}_i(1)+w)$ and $p((r+x)\textbf{e}_i(1)+w,r\textbf{e}_i(1)+w)$ for $\prod_{j\neq i}(n_j+1)$ many values of $w\in\{\sum_{j\ne i}a_{j}\textbf{e}_j(1):a_j\in\mathbb{Z}\}$. So, 
	\begin{equation}\label{lpf1}
	R_i^{r,x}\cdot R_i^{r,x}=2\prod_{j\ne i}(n_j+1).
	\end{equation}
	
	 Every transition probability $p(\textbf{u},\textbf{v})$ involves only one edge parameter $\Gamma(\textbf{u},\textbf{v})$. Hence, $R_i^{r,x}\cdot R_j^{r',y}=0$ for $i\neq j$, $r'\in\{1,2,\dots,n_j-y\}$ and $y\in\{1,2,\dots,l\}$. Here, $R_j^{r',y}$ is the row vector of $\mathcal{R}_{n_1,n_2,\dots,n_q}^l$ corresponding to the edge parameter $\Gamma(r'\textbf{e}_j,(r'+y)\textbf{e}_j)$.
	
	It is clear that $R_{\textbf{u}}\cdot R_i^{r,x}=0$ for all $x\in\{1,2,\dots,l\}$ unless the vector $\textbf{u}$ is either the starting or ending point of the transition that involves the edge parameter corresponding to the vector $R_i^{r,x}$, that is, $\textbf{u}$ belongs to the following set: 
	{\small\begin{equation*}
		\Theta=\Bigg\{r\textbf{e}_i(1)+w:w\in\bigg\{\sum_{j\ne i}a_{j}\textbf{e}_j(1):a_j\in\mathbb{Z}\bigg\}\Bigg\}\bigcup\Bigg\{(r+x)\textbf{e}_i(1)+w:w\in\bigg\{\sum_{j\ne i}a_{j}\textbf{e}_j(1):a_j\in\mathbb{Z}\bigg\}\Bigg\}.
	\end{equation*}}
	 Even for $\textbf{u}\in\Theta$, we have $R_{\textbf{u}}\cdot R_i^{r,x}=(1\cdot1)+(-1)\cdot1=0$, where $1$ and $-1$ are the entries associated to the transition probabilities $p(\textbf{u},\textbf{v})$ and $p(\textbf{v},\textbf{u})$, respectively.
	
	Now, let $\mathscr{U}$ be the submatrix of $\mathcal{R}^l_{n_1,n_2,\dots,n_q}$ that corresponds to the rows associated with the vertex parameters $\alpha_{\textbf{u}}$,  $\textbf{u}\in G^{(q)}_{n_1,n_2,\dots,n_q}$, and let $\mathscr{V}_1,\ \mathscr{V}_2,\dots,\mathscr{V}_q$ are the submatrices of  $\mathcal{R}^l_{n_1,n_2,\dots,n_q}$ that correspond to the rows associated with edge parameters in the direction vectors $\textbf{e}_1(1)$, $ \textbf{e}_2(1)$, $\dots$, $\textbf{e}_q(1)$, respectively. On using (\ref{lpf1}), it can be shown that $\mathscr{V}_i\mathscr{V}_i^T=2\prod_{j\ne i}(n_j+1)I$ for all $i=1,2,\dots,q$ where $I$ is an identity matrix of order equal to total number of edge parameters $\Gamma(r\textbf{e}_i(1),(r+x)\textbf{e}_i(1))$, $r\in\{0,1,\dots,n_i-x\}$, $1\leq x\leq l$. Thus, the rank of matrix $\mathscr{V}_i\mathscr{V}_i^T$ is $\sum_{x=1}^{l}(n_i-x+1)$. Moreover, $\mathscr{V}_i\mathscr{V}_j^T$ is equal to $0$ whenever $i\ne j$ and in view of Proposition \ref{thm3}, we have $\mathscr{U}\mathscr{V}_i^T=0$ for all $i$. Also, we have $\mathscr{U}\mathscr{U}^T=2(\mathscr{D}-\mathscr{A}(G^{(q)}_{n_1,n_2,\dots,n_q}))$, where $\mathscr{D}$ is a diagonal matrix whose diagonal entries are the degrees of corresponding vertices and $\mathscr{A}(G^{(q)}_{n_1,n_2,\dots,n_q})$ is the adjacency matrix of the graph $G^{(q)}_{n_1,n_2,\dots,n_q}$. Hence, $\mathscr{U}\mathscr{U}^T$ is the constant multiple of the Laplacian of $G^{(q)}_{n_1,n_2,\dots,n_q}$. Note that $G^{(q)}_{n_1,n_2,\dots,n_q}$ is a connected graph. From Lemma \ref{plemm1}, the rank of $\mathscr{U}\mathscr{U}^T$ is equal to one less than the total number of vertices, that is, $\mathrm{rank}(\mathscr{U}\mathscr{U}^T)=\prod_{i=1}^{q}(n_i+1)-1$. 
	
	So, we deduce that  $\mathcal{R}^l_{n_1,n_2,\dots,n_q}{\mathcal{R}^l_{n_1,n_2,\dots,n_q}}^T$ is a block diagonal matrix with block matrices $\mathscr{U}\mathscr{U}^T$, $\mathscr{V}_1\mathscr{V}_1^T,\dots,\mathscr{V}_{q-1}\mathscr{V}_{q-1}^T$ and $\mathscr{V}_q\mathscr{V}_{q}^T$. In view of Remark \ref{prem2}, the rank of $\mathcal{R}^l_{n_1,n_2,\dots,n_q}{\mathcal{R}^l_{n_1,n_2,\dots,n_q}}^T$ is given by
	\begin{align*}
		\mathrm{rank}(\mathcal{R}^l_{n_1,n_2,\dots,n_q}{\mathcal{R}^l_{n_1,n_2,\dots,n_q}}^T)&=\sum_{i=1}^{q}\mathrm{rank}({\mathscr{V}_i\mathscr{V}_i^T})+\mathrm{rank}(\mathscr{U}\mathscr{U}^T)\\
		&=\sum_{i=1}^{q}\sum_{x=1}^{l}(n_i-x+1)+\prod_{i=1}^{q}(n_i+1)-1\\
		&=l\sum_{i=1}^{q}n_i-q\sum_{x=1}^{l}(x-1)+\prod_{i=1}^{q}(n_i+1)-1.
	\end{align*} 
	This completes the proof.
\end{proof}
\begin{theorem}
	The rank of matrix $\mathcal{Q}^l_{n_1,n_2,\dots,n_q}$ is
	{\small\begin{equation}\label{r2}
		\mathrm{rank}(\mathcal{Q}^l_{n_1,n_2,\dots,n_q})=2\sum_{i=1}^{q}\sum_{x=1}^{l}(n_i-x+1)\prod_{j\ne i}(n_j+1)-l\sum_{i=1}^{q}n_i-\prod_{i=1}^{q}(n_i+1)+\frac{q(l-1)l}{2}+1.
	\end{equation}}
\end{theorem}
\begin{proof}
	Note that the space $\mathcal{V}_{\mathcal{Q}}$ is a subset of $\mathbb{R}^d$ where $d=2\sum_{i=1}^{q}\sum_{x=1}^{l}(n_i-x+1)\prod_{j\ne i}(n_j+1)$. Thus, the proof follows by using the Proposition \ref{thm3} and Lemma \ref{lemma1}.
\end{proof}

\begin{remark}
	On taking $l=1$ in (\ref{r1}) and (\ref{r2}), we get
	\begin{equation*}
		\mathrm{rank}(\mathcal{R}^1_{n_1,n_2,\dots,n_q})=\sum_{i=1}^{q}n_i+\prod_{i=1}^{q}(n_i+1)-1
	\end{equation*}
	and 
	\begin{equation*}			\mathrm{rank}(\mathcal{Q}^1_{n_1,n_2,\dots,n_q})=2\sum_{i=1}^{q}n_i\prod_{j\ne i}(n_j+1)-\sum_{i=1}^{q}n_i-\prod_{i=1}^{q}(n_i+1)+1.
	\end{equation*}
	These coincide with the results given in (3.5) and (3.6) of Evans \textit{et al.} (2010).
\end{remark}
	
\section{Concluding remarks}
	In this paper, we studied a generalized birth-death process whose state space is a finite dimensional grid. Here, the process can take a jump of finite size unlike the usual birth-death process where there can be a jump of fixed unit size at any point of time.  It is a potentially useful mathematical model for various theoretical and practical problems, for example, in the study of cell replica, in stock markets, in queuing models with multiple servers and bulk arrivals, \textit{etc}. Also, such type of processes are important as their transition probability matrices are diagonalizable, and thus, their finite step transition probability matrices can be efficiently derived. We obtained necessary and sufficient conditions under which the transition probability matrices of each coordinate direction commute with each other.
	
	Future study may include the examination of matrix $\mathcal{R}^l_{n_1,n_2,\dots,n_q}$ from the perspective of commutative algebra (see Sturmfels (1995), Eisenbud and Sturmfels (1996)). As in the case of commuting birth-death process (see Evans \textit{et al.} (2010)), we would like to study the decomposition of space of the generalized birth-death process with positive transition probabilities as a union of toric varieties.

\end{document}